\documentclass[11pt]{amsart}
\usepackage{amsfonts}
\usepackage{amsmath}
\usepackage{amsthm}
\usepackage{url}
\usepackage[all]{xy}
\usepackage{latexsym}
\usepackage{amssymb}
\usepackage[cp850]{inputenc}
\usepackage{amsfonts}
\usepackage{graphicx}
\usepackage{dsfont}
\usepackage{float}
\usepackage{stackrel}

\usepackage[usenames]{color}

\newtheorem{sat}{Theorem}[section]		
\newtheorem{lem}[sat]{Lemma}
\newtheorem{kor}[sat]{Corollary}			
\newtheorem{prop}[sat]{Proposition}

\newtheorem*{defi*}{Definition}			
\newtheorem*{bei*}{Example}
\newtheorem*{sat*}{Theorem}				
\newtheorem*{kor*}{Corollary}
\newtheorem*{rmk*}{Remark}				
	
\newtheorem*{quest*}{Question}

\let\ssection=\section
\renewcommand{\section}{\setcounter{equation}{0}\ssection}

\newtheorem*{namedtheorem}{\theoremname}
\newcommand{\theoremname}{testing}
\newenvironment{named}[1]{\renewcommand{\theoremname}{#1}\begin{namedtheorem}}{\end{namedtheorem}}

\theoremstyle{remark}
\newtheorem*{bem}{Remark}
\newtheorem{bei}{Example}

\newtheorem*{namedtheoremr}{\theoremnamer}
\newcommand{\theoremnamer}{testing}

			\newcommand{\BD}{\mathbb D}

\newcommand{\BI}{\mathbb I}			
			\newcommand{\BL}{\mathbb L}

			\newcommand{\BR}{\mathbb R}
\newcommand{\BS}{\mathbb S}			\newcommand{\BT}{\mathbb T}

			\newcommand{\BZ}{\mathbb Z}

\newcommand{\CC}{\mathcal C}			\newcommand{\calD}{\mathcal D}
\newcommand{\CE}{\mathcal E}

			\newcommand{\CL}{\mathcal L}
\newcommand{\CM}{\mathcal M}		
			\newcommand{\CP}{\mathcal P}

\newcommand{\CU}{\mathcal U}			
		\newcommand{\CX}{\mathcal X}

\newcommand{\FM}{\mathfrak m}

\newcommand{\FP}{\mathfrak p}

\newcommand{\FN}{\mathfrak n}
\newcommand{\FB}{\mathfrak b}

\newcommand{\FC}{\mathfrak c}
\newcommand{\FQ}{\mathfrak q}
\newcommand{\FR}{\mathfrak r}

\newcommand{\actson}{\curvearrowright}
\newcommand{\D}{\partial}

\newcommand{\bs}{\backslash}

\DeclareMathOperator{\Id}{Id}		%	Identit\"at
\DeclareMathOperator{\Map}{Map}

\newcommand{\comment}[1]{}

\DeclareMathOperator{\Stab}{Stab}

\DeclareMathOperator{\Lip}{Lip}
\DeclareMathOperator{\Thu}{Thu}

\DeclareMathOperator{\fil}{fill}

\DeclareMathOperator{\Leb}{Leb}
\DeclareMathOperator{\Ext}{Ext}
\DeclareMathOperator{\polar}{polar}
\DeclareMathOperator{\diag}{diag}
\DeclareMathOperator{\ratio}{ratio}
\DeclareMathOperator{\cone}{cone}

\newcommand{\fsubd}{\mathrel{{\scriptstyle\searrow}\kern-1ex^d\kern0.5ex}}
\newcommand{\bsubd}{\mathrel{{\scriptstyle\swarrow}\kern-1.6ex^d\kern0.8ex}}

\renewcommand{\epsilon}{\varepsilon}
\renewcommand{\le}{\leqslant}
\renewcommand{\ge}{\geqslant}
\renewcommand{\emptyset}{\varnothing}

\begin{document}

\title[]{On the distribution of the components of multicurves of given type}
\author{Viveka Erlandsson}
\address{School of Mathematics, University of Bristol, and}
\address{Department of Mathematics and Statistics, UiT The Arctic University of
Norway}

\email{v.erlandsson@bristol.ac.uk}

\author{Juan Souto}
\address{IRMAR, Universit\'e de Rennes}
\email{juan.souto@univ-rennes.fr}

\begin{abstract}
We study the distribution of the individual components of a random multicurve under the action of the mapping class group.
\end{abstract}

\maketitle

\section{Introduction}
Throughout this paper let $\Sigma$ be a compact, connected, orientable surface of genus $g$ and $r$ boundary components satisfying $2g+r\ge 3$. We denote by $\Vert\cdot\Vert$ the $\ell^1$-norm on Euclidean space and by $\Delta_k=\{\vec x\in\BR^k_{\ge 0},\ \Vert\vec x\Vert=1\}$ the standard simplex---note that the dimension of $\Delta_k$ is $k-1$. Finally,  independently of the concrete ambient space, we denote by $\delta_p$ the Dirac probability measure centered on $p$.
\medskip

\subsection*{Distribution of vectors of lengths and ratios}
Under a $k$-multicurve on $\Sigma$ we consider an ordered tuple $\vec\gamma=(\gamma_1,\dots,\gamma_k)$ of (primitive, essential) curves, weighted if you so wish, or even multicurves, that is, a formal linear combination of curves with positive coefficients. If $\phi$ is any reasonable notion of length, then we will use polar coordinates to represent the vector encoding the lengths $\phi(\gamma_1),\dots,\phi(\gamma_k)$ of the individual components:
$$\BL_\phi(\vec\gamma)=\left(\frac 1{\sum_i \phi(\gamma_i)}(\phi(\gamma_1),\dots,\phi(\gamma_k)),\sum_i \phi(\gamma_i)\right)\in\Delta_k\times\BR_{\ge 0}.$$
The mapping class group $\Map(\Sigma)$ acts on the set of $k$-multicurves and the goal of this paper is to look at the distribution of the vectors $\BL_\phi(\vec\gamma)$ where $\vec\gamma$ ranges over the mapping class group orbit $\Map(\Sigma)\cdot\vec\gamma^o$ of some fixed $k$-multicurve. More specifically we will study the asymptotic behavior of the measures
$$\underline m(\phi,\vec\gamma^o,L)=\frac 1{\vert{\bf M}_{\phi,\vec\gamma^o}(L)\vert}\sum_{\vec\gamma\in\Map(\Sigma)\cdot\vec\gamma^o}\delta_{\BL_\phi(\frac 1L \vec\gamma)}$$
on $\Delta_k\times\BR_{\ge 0}$, where 
$${\bf M}_{\phi,\vec\gamma^o}(L)=\left\{(\gamma_i)_i\in\Map(\Sigma)\cdot\vec\gamma^o\ \middle\vert\  \sum_i\phi(\gamma_i)\le L\right\}.$$
This problem was first studied by Mirzakhani in \cite{Maryam general} in the case that $\phi=\ell$ is the length function of a hyperbolic structure on $\Sigma$ and that $\vec\gamma^o$ is a pants decomposition. Later, building on Mirzakhani's work, Arana-Herrera \cite{Arana} and Liu \cite{Liu} studied the case that $\vec\gamma^o$ is a general simple $k$-multicurve, that is a $k$-multicurve whose components are simple and disjoint of each other. For example, they proved that for simple $\vec\gamma^o$ the measures $\underline m(\ell,\vec\gamma^o,L)$ have a limit when $L\to\infty$, and that in fact the limit does not depend on the chosen hyperbolic metric. Our first result is that this is true for all $k$-multicurves, simple or not, and all $\phi$ in a very large class of length functions:

\begin{sat}\label{main counting}
For every $k$-multicurve $\vec\gamma^o=(\gamma_1,\dots,\gamma_k)$ in $\Sigma$ there is a probability measure $\FP_{\vec\gamma^o}$ on $\Delta_k$ such that for any continuous, positive, homogenous function $\phi:\CC(\Sigma)\to\BR_{\ge 0}$ on the space of currents on $\Sigma$ we have
$$\lim_{L\to\infty}\underline m(\phi,\vec\gamma^o,L)=\FP_{\vec\gamma^o}\otimes\left((6g-6+2r)\cdot t^{6g-7+2r}\cdot d{\bf t}\right)$$
where $d{\bf t}$ stands for the standard Lebesgue measure on $\BR_{\ge 0}$. In particular, the limiting distribution is independent of $\phi$.
\end{sat}

Since $\FP_{\vec\gamma^o}(U)=\big(\FP_{\vec\gamma^o}\otimes(6g-6+2r)\cdot t^{6g-7+2r}d{\bf t}\big)(U\times [0,1])$ we get the following from Theorem \ref{main counting}:

\begin{kor}\label{kor counting}
With notation as in Theorem \ref{main counting} we have 
$$\lim_{L\to\infty}\frac{\vert\{\vec\gamma\in{\bf M}_{\phi,\gamma^o}(L)\text{ with } \frac 1{\sum_i\phi(\gamma_i)}(\phi(\gamma_i))_i\in U\}\vert}{\vert{\bf M}_{\phi,\vec\gamma^o}(L)\vert}=\FP_{\vec\gamma^o}(U)$$
for every $U\subset\Delta_k$ open with $\FP_{\vec\gamma^o}(\D\overline U)=0$.\qed
\end{kor}

In some cases the measure $\FP_{\vec\gamma^o}$ can be explicitly calculated and this is actually what Mirzakhani, Arana-Herrera, and Liu do for simple $k$-multicurves. For example, if $\vec P^o$ is a pants decomposition we have
\begin{equation}\label{eq p pants}
\FP_{\vec P^o}(U) = \frac{(6g-7+2r)!}{\sqrt{3g-3+r}}\cdot\int_U \prod x_i\ d\Leb_{\Delta_{3g-3+r}}
\end{equation}
where $\Leb_{\Delta_{3g-3+r}}$ is the Lebesgue measure on the simplex (see Example \ref{example pants} in Section \ref{sec examples} below). From here we get statements like the following: for any continuous, positive and homogenous function $\phi:\CC(\Sigma)\to\BR_{\ge 0}$ and any $\epsilon$ we have that
$$\lim_{L\to\infty}\frac{\vert\{\vec P\in{\bf M}_{\phi,\vec P^o}(L)\vert\text{ there is }i\text{ with }\phi(\gamma_i)\leq\epsilon\cdot\frac{L}{3g-3+r}\}\vert}{\vert{\bf M}_{\phi,\vec P^o}(L)\vert}\leq\epsilon$$
for any pants decomposition $\vec P^o$. It might be worth noticing that for $\epsilon$ very small, this bound is pretty course since we can also bound the limit from above by $(6g-6+2r)\cdot\epsilon^2$ which in fact is asymptotically correct as $\epsilon\to 0$.  
\medskip

Calculating the measure $\FP_{\vec\gamma^o}$ for a general $\vec\gamma^o$, simple or not, is in principle feasible (see Proposition \ref{prop describe p measure}), but it becomes very fast very complicated. Deep down, the reason why $\FP_{\vec\gamma^o}$ is in principle computable is because it is the push-forward under a polyhedral map $\BS^{6g-7+2r}\to\Delta_k$ of a measure in the Lebesgue class with PL-density. On the other hand, there is no bound independent of $\vec\gamma^o$ for the complexity of the polyhedral map we just mentioned (see Example \ref{example p complicated} in Section \ref{sec examples}).
\medskip

Recall now that there is a plethora of continuous, homogenous and positive functions on the space of currents \cite{Viv word length,Hugo,Didac-Dylan} such as for example the length function associated to any negatively curved Riemannian metric, the extremal length associated to any conformal structure, the word length on $\pi_1(\Sigma)$ associated to certain specific sets of generators, the stable length associated to any discrete and cocompact action of $\pi_1(\Sigma)$ on a geodesic metric space, etc, etc... Still, the conclusion of Theorem \ref{main counting} is that the asymptotic distribution of the vectors $\BL_{\phi}(\vec\gamma)$ for $\vec\gamma$ in a mapping class group orbit does not change if we replace $\phi$ by another function $\psi$. On the other hand it is very sensitive to the choice of $\vec\gamma^o$. Now, the roles of these two factors flip when we consider the distribution
\begin{equation}\label{eq distribution L ratios}
r(\psi/\phi,\vec\gamma^o,L)=\frac 1{\vert{\bf M}_{\phi,\vec\gamma^o}(L)\vert}\sum_{\vec\gamma\in{\bf M}_{\phi,\vec\gamma^o}(L)}\delta_{\ratio_{\psi/\phi}(\vec\gamma)}
\end{equation}
of the vectors 
$$\ratio_{\psi/\phi}((\gamma_i)_i)=\left(\frac{\psi(\gamma_i)}{\phi(\gamma_i)}\right)_i\in\BR_{\ge 0}^k$$
whose entries are the ratios of the $\psi$- and $\phi$-lengths of the individual components of the $k$-multicurve in question.

\begin{sat}\label{main ratios}
For any two continuous, homogenous, positive functions $\phi,\psi:\CC(\Sigma)\to\BR_{\ge 0}$ there is a probability measure $\FR_{\psi/\phi}$ on $\BR_{\ge 0}$ with
$$\diag_*(\FR_{\psi/\phi})=\lim_{L\to\infty}\frac 1{\vert{\bf M}_{\phi,\vec\gamma^o}(L)\vert} r(\psi/\phi,\vec\gamma^o,L)$$
for any $k$-multicurve $\vec\gamma^o$. Here $\diag:\BR_{\ge 0}\to\BR_{\ge 0}^k$ is the diagonal map $\diag(t)=(t,\dots,t)$.
\end{sat}

\begin{bem}
In \eqref{eq distribution L ratios} we take a sum over the subset ${\bf M}_{\phi,\vec\gamma^o}(L)\subset\Map(\Sigma)\cdot\vec\gamma^o$ because we want to be working with probability measures, or at least with Radon measures. If we had summed over the whole orbit then we would have obtained measures which are not locally finite.
\end{bem}

While the measure $\FP_{\vec\gamma^o}$ from Theorem \ref{main counting} is rather well-behaved, the measure $\FR_{\psi/\phi}$ from Theorem \ref{main ratios} can present all sorts of pathologies. For example, while its support is always a (possibly degenerate) interval in $\BR_{>0}$, it can well have infinitely many atoms (see Example \ref{example r atoms} in Section \ref{sec examples}).

\subsection*{Distribution of multicurves}
While Mirzakhani, Arana-Herrera, and Liu rely on the dynamics of the earthquake flow, we use a rather different approach. Thinking of multicurves as currents, we think of $k$-multicurves as $k$-multicurrents, that is, as elements in the product $\CC(\Sigma)^k=\CC(\Sigma)\times\dots\times\CC(\Sigma)$, and we study the distribution of the orbit $\Map(\Sigma)\cdot\vec\gamma^o$ in $\CC(\Sigma)^k$. More precisely, we will study the measures 
$$m(\vec\gamma^o,L)=\frac 1{L^{6g-6+2r}}\sum_{\vec\gamma\in\Map(\Sigma)\cdot\vec\gamma^o}\delta_{\frac 1L\vec\gamma}$$
as $L\to\infty$. Our main theorem, Theorem \ref{main} below, asserts that these measures converge to some measure as $L\to\infty$, but before being able to state it we need some notation. First, consider the map
\begin{equation}\label{eq multidim inter}
\BI:\CC(\Sigma)\times\CC(\Sigma)^k\to\BR_{\ge 0}^k,\ \ \ \BI:(\sigma,(\mu_i)_i)\mapsto(\iota(\sigma,\mu_i))_i
\end{equation}
where $\iota(\cdot,\cdot)$ is the geometric intersection form. Now, given a $k$-multicurve $\vec\gamma^o$, denote by $\calD_{\vec\gamma^o}$ some fundamental domain for the action $\Stab(\vec\gamma^o)\actson\CM\CL(\Sigma)$ of the stabilizer of $\vec\gamma^o$ on the space of measured laminations (see Proposition \ref{prop fundamental domain} for details) and consider the measure $\FQ_{\vec\gamma^o}$ on $\Delta_k$ given by 
\begin{equation}\label{eq q measure}
\FQ_{\vec\gamma^o}(U)=\frac{1}{\FB_{g,r}}\cdot\BI(\cdot,\vec\gamma^o)_*(\FM_{\Thu}\vert_{\calD_{\vec\gamma^o}})(\cone(U))
\end{equation}
where $\FM_{\Thu}$ is the Thurston measure on $\CM\CL(\Sigma)$, 
%where $$\cone(U)=\{tu\text{ with }t\in[0,1],u\in U\}$$  is the cone with basis $U$ and 
where $\FB_{g,r}$ is a constant that depends only on the topological type of the surface (see Section \ref{subsec curve counting} for details), and where $\cone(U)=\{tu\text{ with }t\in[0,1],u\in U\}$. Consider also the map
$$\BD:\Delta\times\CM\CL(\Sigma)\to\CC(\Sigma)^k,\ \ \BD:((a_i),\lambda)\mapsto(a_1\lambda,\dots,a_k\lambda),$$
of which we think as a thickened diagonal embedding. With this notation in place we have:

\begin{sat}\label{main}
For any $k$-multicurve $\vec\gamma^o=(\gamma_i)_i$ in $\Sigma$, we have 
\begin{equation}\label{eq main limit}
\BD_*(\FQ_{\vec\gamma^o}\otimes\FM_{\Thu})=\lim_{L\to\infty}m(\vec\gamma^o,L)
\end{equation}
in the weak-*-topology on Radon measures on $\CC(\Sigma)^k$.
\end{sat}

To prove Theorem \ref{main} we will rely on the measure convergence theorem we proved in \cite{ES book} (see Section \ref{subsec curve counting} below), and in fact we will also coarsely follow the strategy of its proof. First we show that the family of measures $(m(\vec\gamma^o,L))_L$ is precompact and that every sublimit $\FN$ as $L\to\infty$ is supported by the image of $\BD$. We then use the measure convergence theorem from \cite{ES book} to prove that any such sublimit is of the form $\FN=\BD_*(\FQ\otimes\FM_{\Thu})$ for some measure $\FQ$ on $\Delta_k$. Then we prove a specific version of a weaker form of Theorem \ref{main counting}, one for the function $\phi(\vec\gamma)=\BI(\sigma^o,\vec\gamma)$ for some filling multicurve satisfying a certain technical condition. We can use this version of Theorem \ref{main counting} to identify the measure $\FQ$, proving that it is identical to the measure $\FQ_{\vec{\gamma}^o}$ given in \eqref{eq q measure}. Since the measure $\FQ$, and hence the sublimit $\FN$, does not depend on the specific sequence used to extract the sublimit, we get that the actual limit exits. This concludes the sketch of the proof of Theorem \ref{main}.
\medskip

Once we have proved Theorem \ref{main}, the earlier mentioned results follow easily. For example, the measure $\FP_{\vec\gamma^o}$ from Theorem \ref{main counting} is the probability measure associated to $\FQ_{\vec\gamma^o}$:
$$\FP_{\vec\gamma^o}=\frac 1{\Vert\FQ_{\vec\gamma^o}\Vert}\cdot\FQ_{\vec\gamma^o},$$
where, in the context of measures, $\Vert\cdot\Vert$ stands for the total measure. The measure $\FR_{\psi/\phi}$ is, up to scaling, nothing other than the push-forward under $\lambda\mapsto\frac{\psi(\lambda)}{\phi(\lambda)}$ of the restriction of $\FM_{\Thu}$ to the set of measured laminations with $\phi(\lambda)\le 1$:
$$\FR_{\psi/\phi}=\frac 1{\FM_{\Thu}(\{\lambda\in\CM\CL,\ \phi(\lambda)\le 1\})}\cdot\left(\frac{\psi(\cdot)}{\phi(\cdot)}\right)_*\big(\FM_{\Thu}\vert_{\{\lambda\in\CM\CL,\ \phi(\lambda)\le 1\}}\big).$$
\medskip

This paper is organized as follows: After some preliminaries in Section \ref{sec preli} we prove in Section \ref{sec first convergence} the weaker form of a version Theorem \ref{main counting} mentioned earlier. In Section \ref{sec subconvergence} we prove the subconvergence of the measures $m(\vec\gamma^o,L)$ as $L\to\infty$ and that every sublimit is of the form $\BD_*(\FQ\otimes\FM_{\Thu})$ for some measure $\FQ$ on $\Delta_k$. Combining the results of Section \ref{sec first convergence} and Section \ref{sec subconvergence} we prove Theorem \ref{main} in Section \ref{sec main}. Theorem \ref{main counting} and Corollary \ref{kor counting} are proved in Section \ref{sec counting} and Theorem \ref{main ratios} in Section \ref{sec ratios}. In the final Section \ref{sec examples} we discuss the measures $\FP_{\vec\gamma^o}$ and $\FR_{\psi/\psi}$ in some concrete cases.

\subsection*{Acknowledgements} We thank Mingkun Liu for helpful discussions: without his help we would still be computing integrals.

\section{Some background}\label{sec preli}
As in the introduction, let $\Sigma$ be a compact, connected and orientable surface of genus $g$ with $r$ boundary components satisfying $2g+r\ge 3$. We also fix some arbitrary hyperbolic metric on $\Sigma$ with respect to which the boundary is totally geodesic. As is customary, we identify curves with their free homotopy classes and with closed geodesics in $\Sigma$. We will always assume that curves are primitive and that they are essential in the sense that they cannot be homotoped into $\D\Sigma$, nor into a point. We stress that we do {\em not} assume that curves are simple.

\subsection*{Warning} Although the results from the introduction are valid in all generality, we will assume from now on that $\Sigma$ is {\em non-exceptional}, meaning that $(g,r)$ is not one of the following: $(0,3),(1,1),(1,2)$ and $(2,0)$. In other words, we are ruling out those surfaces where the mapping class group is either finite or has non-trivial center. The center does not cause any fundamental difficulties, but it involves further complicating an already pretty exuberant notation. In any case, the reader who wants to consider one of the exceptional surfaces will have no difficulty modifying the arguments. At the end of the day we rely heavily on \cite{ES book}, where we explicitly dealt with the center. No center from now on!

\subsection{Currents}
Currents on $\Sigma$ are nothing other than $\pi_1(\Sigma)$-invariant Radon measures on the set of geodesics on the universal cover of $\Sigma$. We endow the space $\CC(\Sigma)$ of currents with the weak-*-topology. The set $\CC_K(\Sigma)$ of currents supported by the set of geodesics contained in the preimage of some compact $K\subset\Sigma\setminus\D\Sigma$ is closed. On the other hand, the set $\CC_0(\Sigma)=\cup_K\CC_K(\Sigma)$ of compactly supported currents is not. Currents were introduced by Bonahon in \cite{Bonahon86,Bonahon88,Bonahon90}, but instead of referring to the original papers or to other sources (we still recommend the very nice paper \cite{Javi-Chris}) we will mostly follow the terminology and notation from \cite{ES book}. 
\medskip

When working with currents, it often does not really matter what currents actually are: what matters is that they exist, that some well known objects are currents, and that the space of currents has some very nice properties. We summarize what we will need here:
\medskip

\noindent{{\bf (1)}
The space $\CC(\Sigma)$ of currents is a locally compact metrizable topological space. It is moreover a linear cone in a topological vector space, meaning that one can add currents and that one can multiply currents by positive numbers, and that these two operations are associative, distributive, and continuous.
\medskip

\noindent{{\bf (2)}
Curves and measured laminations are currents. In fact, the embedding $\CM\CL(\Sigma)\hookrightarrow\CC(\Sigma)$ is a homeomorphism onto its image, a subset of $\CC_K(\Sigma)$ for some suitable compact set $K\subset\Sigma\setminus\D\Sigma$. On $\CM\CL(\Sigma)$ there is a natural $\Map(\Sigma)$-invariant measure, the Thurston measure $\FM_{\Thu}$, which scales with power $6g-6+2r$, that is $\FM_{\Thu}(t\cdot U) = t^{6g-6+2r}\cdot\FM_{\Thu}(U)$. When we push forward $\FM_{\Thu}$ by the embedding from $\CM\CL(\Sigma)$ to $\CC(\Sigma)$ we get a measure to which we still refer to by the same name and which we denote by the same symbol $\FM_{\Thu}$.  Masur proved that the set of uniquely ergodic  measured laminations which intersect every curve has full Thurston measure \cite{Masur almost all uq} and that $\FM_{\Thu}$ is ergodic with respect to $\Map(\Sigma)$ action \cite{Masur ergodic}. Note that there are different reasonable normalizations for the Thurston measure---here we follow the one coming from constructing it as a scaling limit (see \cite[Theorem 4.16]{ES book}). \medskip

\noindent{{\bf (3)}
There is a continuous bilinear map
      $$\iota:\CC(\Sigma)\times\CC(\Sigma)\to\BR_{\ge 0},$$ 
the {\em geometric intersection form}, extending the geometric intersection number of both curves and measured laminations. Indeed, measured laminations are characterized as being exactly those currents $\lambda$ with vanishing self-intersection number $\iota(\lambda,\lambda)=0$.
\medskip

\noindent{{\bf (4)}
A compactly supported current $\mu\in\CC_0(\Sigma)$ is {\em filling} if we have $\iota(\mu,\lambda)>0$ for every non-zero compactly supported current $\lambda\in\CC_0(\Sigma)$. We denote the set of all filling currents by $\CC_{\fil}(\Sigma)$. A multicurve is filling if the associated current is. In topological terms, a multicurve is filling if and only if the geodesic representative of its support cuts the surface into a collection of disks and annuli, each one of the latter containing a boundary component.

A key property of filling currents is the following: Given $K\subset\Sigma\setminus\D\Sigma$ compact and $\mu\in\CC(\Sigma)$ filling, then the set $\{\lambda\in\CC_K(\Sigma)\text{ with }\iota(\mu,\lambda)\le L\}$ is compact for all $L$.
\medskip

\noindent{{\bf (5)} 
The mapping class group acts continuously on $\CC(\Sigma)$ by linear automorphisms preserving the intersection form. For every curve $\gamma$ there is a compact set $K\subset\Sigma\setminus\D\Sigma$ with $\Map(\Sigma)\cdot\gamma\subset\CC_K(\Sigma)$. 
\medskip

Besides the basic properties of currents and the space of currents that we just listed, the following theorem of Mondello and the first author \cite{Mondello} plays an important role here:

\begin{sat*}[Erlandsson-Mondello]
The action of $\Map(\Sigma)$ on the open set $\CC_{\fil}(\Sigma)$ of filling currents is discrete in the sense that for every $A\subset\CC_{\fil}(\Sigma)$ compact there are only finitely many $\phi\in\Map(\Sigma)$ with $\phi(A)\cap A\neq\emptyset$.
\end{sat*}

In \cite{ES book} we explained how to construct fundamental domains for the action of certain subgroups of the mapping class group on certain subspaces of the space of filling currents. Below we will be working in exactly the same setting as there but, since nothing changes, it might be a good idea to state it in generality. 

So, let us suppose that $\Gamma\subset\Map(\Sigma)$ is a subgroup and $\CU\subset\CC_{\fil}(\Sigma)$ is a $\Gamma$-invariant subset of the space of filling currents. Our aim is to construct a fundamental domain for the action $\Gamma\actson\CU$. To do so let us fix a filling multicurve $\sigma^o$ satisfying
\begin{equation}\label{eq technical condition}
\FM_{\Thu}(\{\lambda\in\CM\CL(\Sigma)\text{ with }\iota(\sigma^o,\lambda)=\iota(\sigma^o,\phi(\lambda))\})=0
\end{equation}
for every $\phi\in\Map(\Sigma)$---we proved in \cite[Lemma 8.3]{ES book} that such multicurves exist. Once we have fixed $\sigma^o$ we can consider the set 
\begin{equation}\label{eq fundamental domain}
\calD_{\Gamma}=\{\lambda\in\CU\text{ with }\iota(\lambda,\sigma^o)<\iota(\lambda,\phi(\sigma^o))\text{ for all }\phi\in\Gamma\setminus\Id\}.
\end{equation}
Repeating word-by-word the argument of \cite[Proposition 8.4]{ES book} we get that $\calD_\Gamma$ is a fundamental domain in the following sense:

\begin{prop}\label{prop fundamental domain}
The set $\calD_{\Gamma}$ is open in $\CU$ and satisfies $\calD_{\Gamma}\cap\phi(\calD_{\Gamma})=\emptyset$ for all $\phi\in\Gamma\setminus\{\Id\}$. On the other hand, the translates of its closure $\overline\calD_\Gamma$ cover $\CU$, that is $\CU=\cup_{\phi\in\Gamma}\phi(\overline\calD_\Gamma)$. Finally, we also have that $\overline\calD_\Gamma\setminus\calD_\Gamma$ has vanishing Thurston measure: $\FM_{\Thu}(\overline\calD_\Gamma\setminus\calD_\Gamma)=0$.\qed
\end{prop}

\subsection{Curve counting}\label{subsec curve counting}
Given a multicurve $\gamma^o$ in $\Sigma$, Mirzakhani studied in \cite{Maryam simple, Maryam general} the asymptotic growth of the cardinality of the set
$${\bf M}_{\Sigma,\gamma^o}(L)=\{\gamma\in\Map(\Sigma)\cdot\gamma^o\text{ with }\ell_\Sigma(\gamma)\le L\}$$
where $\ell_\Sigma$ stands for the hyperbolic length. She proved in \cite{Maryam simple} for simple multicurves and in \cite{Maryam general} for general ones that the limit
\begin{equation}\label{eq maryam counting}
\lim_{L\to\infty}\frac{1}{L^{6g-6+2r}}\vert{\bf M}_{\Sigma,\gamma^o}(L)\vert
\end{equation}
exists and is positive. 

\begin{bem}
In fact Mirzakhani gave the value for the limit above but we prefer to not do it explicitly now: we are going to express it shortly in a somewhat different way than she did and we do not want to risk causing any confusion.
\end{bem}

In \cite{ES book} we obtained the existence of \eqref{eq maryam counting} from a measure convergence theorem (as did Mirzakhani already in \cite{Maryam simple} in the case that $\gamma^o$ is simple). More precisely, denoting as always by $\delta_x$ the Dirac probability measure centered on $x$, we considered the measures
\begin{equation}\label{eq counting measure}
m^L_{\gamma^o}=\frac 1{L^{6g-6+2r}}\sum_{\gamma\in\Map(\Sigma)\cdot\gamma^o}\delta_{\frac 1L\gamma}
\end{equation}
on the space of currents $\CC(\Sigma)$ and proved that they converge when $L\to\infty$:

\begin{sat*}\cite[Theorem 8.1]{ES book}
Given any multicurve $\gamma^o\in\CC(\Sigma)$ let $\calD_{\Stab(\gamma^o)}$ be a fundamental domain for the action of $\Stab(\gamma^o)$ on the set 
$$\CC_{\gamma^o}=\{\lambda\in\CC_0(\Sigma)\text{ with }\lambda+\gamma^o\text{ filling}\}$$ 
and set
$$\FC(\gamma^o)=\FM_{\Thu}(\{\lambda\in\calD_{\Stab(\gamma^o)},\ \iota(\gamma^o,\lambda)\le 1\}).$$
Then we have
\begin{equation}\label{eq measure convergence}
\lim_{L\to\infty}m^L_{\gamma^o}=\frac{\FC(\gamma^o)}{\FB_{g,r}}\cdot\FM_{\Thu}.
\end{equation}
where $\FB_{g,r}$ depends just on the genus and number of boundary components of $\Sigma$.
\end{sat*}

\begin{bem}
The constant $\FB_{g,r}$ in the theorem can actually be expressed in the following two ways.
In \cite{ES book} we defined it as 
\begin{equation*}
\FB_{g,r}
=\sum_{\gamma\in\Map(\Sigma)\bs\CM\CL_\BZ(\Sigma)}\FC(\gamma)
\end{equation*}
where $\CM\CL_\BZ(\Sigma)$ is the set of all simple integrally weighted multicurves,
but closer to what Mirzakhani did in \cite{Maryam simple}, it can also be expressed as
\begin{equation*}
\FB_{g,r}=2^{2g-3+r}\cdot\int_{\CM_{g,r}}\FM_{\Thu}(\{\lambda\in\CM\CL(\Sigma)\text{ with }\ell_X(\lambda)\le 1\})\ dX
\end{equation*}
 where the integral is taken with respect to the Weil-Petersson metric on $\CM_{g,r}$, the moduli space of hyperbolic surfaces of genus $g$ and $r$ cusps. See \cite[Proposition 8.8]{ES book} for the relationship between our and Mirzakhani's constants. 
\end{bem}

From the measure convergence theorem above we obtained that for any positive, homogenous and continuous function $\phi:\CC(\Sigma)\to\BR_{\ge 0}$ we have convergence when in \eqref{eq maryam counting} we replace the hyperbolic length function $\ell_\Sigma$ by the function $\phi$. This applies in particular to functions like $\lambda\mapsto\iota(\sigma^o,\lambda)$ for some filling current $\sigma^o$. Since this is the setting we are going to be working in we state this precisely. For a filling current $\sigma^o$ consider the set
\begin{equation}\label{eq set counted}
{\bf M}_{\sigma^o,\gamma^o}(L)=\{\gamma\in\Map(\Sigma)\cdot\gamma^o\text{ with }\iota(\sigma^o,\gamma)\le L\}.
\end{equation}
Then we have
\begin{equation}\label{eq counting}
\lim_{L\to\infty}\frac{1}{L^{6g-6+2r}}\vert{\bf M}_{\sigma^o,\gamma^o}(L)\vert=\frac{\FC(\gamma^o)}{\FB_{g,r}}\cdot B(\sigma^o)
\end{equation}
where $B(\sigma^o)=\FM_{\Thu}(\{\lambda\in\CM\CL(\Sigma)\text{ with }\iota(\sigma^o,\lambda)\le 1\})$.

\subsection{$k$-multicurrents}
For lack of a better name, we refer to the elements $\vec\mu=(\mu_1,\dots,\mu_k)$ in $\CC(\Sigma)^k$ as {\em $k$-multicurrents}. Note for example that $k$-multicurves are $k$-multicurrents. As the reader has already seen, we indicate such $k$-multicurves and $k$-multicurrents either by writing out the tuple, maybe in the short form $(\mu_i)_i$, or by decorating them with a little arrow on top. The mapping class group acts diagonally on the space of $k$-multicurrents:
$$\phi((\mu_i)_i)=(\phi(\mu_i))_i$$
The mapping class group equivariant map
$$\pi:\CC(\Sigma)^k\to\CC(\Sigma),\ \ \pi((\mu_i)_i)=\sum_i\mu_i$$
will play a key role in this paper, but before going any further let us agree on some terminology: we will refer to the entries $\mu_1,\dots,\mu_k$ of a $k$-multicurrent $\vec\mu=(\mu_1,\dots,\mu_k)\in\CC^k(\Sigma)$ as its {\em components} and we will say that the $k$-multicurrent $\vec\mu$ {\em represents} the current $\pi(\vec\mu)$. Note that a $k$-multicurrent which represents a multicurve is nothing other than a $k$-multicurve.

Although it is very important that we work with general currents, we will be especially interested on those $k$-multicurrents representing measured laminations. Let us prove two simple lemmas that will look self-evident to experts:

\begin{lem}\label{lem1}
A $k$-multicurrent in $\vec\mu=(\mu_1,\dots,\mu_k)\in\CC^k(\Sigma)$ represents a measured lamination if and only if all the components $\mu_i$ are also measured laminations satisfying $\iota(\mu_i,\mu_j)=0$ for all $i$ and $j$.
\end{lem}
\begin{proof}
By assumption we have that 
$$0=\iota\left(\pi(\vec\mu),\pi(\vec\mu)\right)=\iota\left(\sum_i\mu_i,\sum_j\mu_j\right)=\sum_{i,j}\iota(\mu_i,\mu_j)$$
and since the intersection form only takes non-negative values it follows that $\iota(\mu_i,\mu_j)=0$ for all $i,j$. Now the claim follows from the fact that measured laminations are characterized as being the only currents with vanishing self-intersection number.
\end{proof}

Recall now that a measured lamination is {\em uniquely ergodic} if its support only carries a single transverse measure up to scaling.

\begin{lem}\label{lem2}
If a $k$-multicurrent $\vec\mu\in\CC^k(\Sigma)$ represents an uniquely ergodic measured lamination $\lambda$ then there is a point $(a_i)_i\in\Delta_k$ in the standard $k$-simplex with $\vec\mu=(a_1\lambda,\dots,a_k\lambda)$.
\end{lem}

\begin{proof}
Each component $\mu_i$ of $\vec\mu$ represents a measured lamination with support contained in that of $\lambda$. It follows thus from the assumption that $\lambda$ is uniquely ergodic that there is some $a_i\ge 0$ with $\mu_i=a_i\lambda$. The assumption that the vector $(a_1\lambda,\dots,a_k\lambda)$ represents $\lambda$ implies that $\sum a_i=1$, meaning that $(a_1,\dots,a_k)\in\Delta_k$.
\end{proof}

\section{A first convergence result}\label{sec first convergence}
Still with the same notation let $\Sigma$ be a compact orientable connected surface with genus $g$ and $r$ boundary components, and recall that we are just dealing with the case that $\Sigma$ is non-exceptional. In this section we prove a specific case of a close relative of Theorem \ref{main counting}, but first let us recall the map \eqref{eq multidim inter} from the introduction:
$$\BI:\CC(\Sigma)\times\CC(\Sigma)^k\to\BR_{\ge 0}^k,\ \BI(\sigma,(\mu_i)_i)=(\iota(\sigma,\mu_i)_i).$$
This is what we prove:

\begin{prop}\label{prop special case}
Let $\vec\gamma^o$ be a $k$-multicurve in $\Sigma$, let $\sigma^o$ be a filling multicurve in $\Sigma$ satisfying \eqref{eq technical condition}, and let $\calD=\calD_{\Stab(\vec\gamma^o)}$ be as in \eqref{eq fundamental domain}. We have
$$\BI(\cdot,\vec\gamma^o)_*\left(\frac{B(\sigma^o)}{\FB_{g,r}}\cdot(\FM_{\Thu}\vert_{\calD})\right)=\lim_{L\to\infty}\frac 1{L^{6g-6+2r}}\sum_{\vec\gamma\in\Map(\Sigma)\cdot\vec\gamma^o}\delta_{\frac 1L \BI(\sigma^o,\vec\gamma)}$$
in the weak-*-topology on measures on $\BR_{\ge 0}^k$. Here $B(\sigma^o)=\FM_{\Thu}(\{\lambda\in\CM\CL(\Sigma)\text{ with }\iota(\sigma^o,\lambda)\le 1\})$.
\end{prop}

\begin{proof}
Denote the measures inside the limit by
$$\overline m(\sigma^o,\vec\gamma^o,L)=\frac 1{L^{6g-6+2r}}\sum_{\vec\gamma\in\Map(\Sigma)\cdot\vec\gamma^o}\delta_{\frac 1L(\BI(\sigma^o,\vec\gamma))}.$$
Note that these measures are supported on $\mathbb{R}_{\geq0}^k$ as opposed to $\Delta_k\times\mathbb{R}_{\geq0}$ which is the case for the measures $\underline m(\phi, \vec\gamma^0, L)$ in Theorem \ref{main counting}, and this is the reason why we use a different decoration on $m$.

Since we have $\BI(\sigma^o,\phi(\vec\gamma))=\BI(\phi^{-1}(\sigma^o),\vec\gamma)$ we can, for any choice of a set $\Theta\subset\Map(\Sigma)$ of representatives of classes $\Map(\Sigma)/\Stab(\vec\gamma^o)$, rewrite them as
$$\overline m(\sigma^o,\vec\gamma^o,L)=\BI(\cdot,\vec\gamma^o)_*\left(\frac 1{L^{6g-6+2r}}\sum_{\phi\in\Theta}\delta_{\frac 1L \phi^{-1}(\sigma^o)}\right).$$
We now get from Proposition \ref{prop fundamental domain} that $\Theta$ can be chosen to satisfy
$$\{\phi\in\Map(\Sigma)\vert\phi^{-1}(\sigma^o)\in\calD\}\subset\Theta\subset\{\phi\in\Map(\Sigma)\vert\phi^{-1}(\sigma^o)\in\overline\calD\}.$$
Indeed, the first inclusion follows from the fact that every $\Stab(\vec\gamma^o)$-orbit meets $\calD$ at most once, and the second because $\overline\calD$ meets every orbit at least once. Anyways, from here we get the following comparisons for the measure $\overline m(\sigma^o,\vec\gamma^o,L)$:
\begin{equation}\label{ch10 equation have to go now1}
\begin{split}
\overline m(\sigma^o,\vec\gamma^o,L)\ge
\BI(\cdot,\vec\gamma^o)_*\left(\frac 1{L^{6g-6+2r}}\sum_{
\sigma\in(\Map(\Sigma)\cdot \sigma^o)\cap\calD
}\delta_{\frac 1L\sigma}\right)
\\
\overline m(\sigma^o,\vec\gamma^o,L)\le
\BI(\cdot,\vec\gamma^o)_*\left(\frac 1{L^{6g-6+2r}}\sum_{
\sigma\in(\Map(\Sigma)\cdot\sigma^o)\cap\overline\calD
}\delta_{\frac 1L\sigma}\right)
\end{split}
\end{equation}
From \eqref{eq measure convergence} applied to $\sigma^o$, we get that
$$\lim_{L\to\infty}\frac 1{L^{6g-6+2r}}\sum_{\sigma\in\Map(\Sigma)\cdot\sigma^o}\delta_{\frac 1L\sigma}=\frac{\FC(\sigma^o)}{\FB_{g,r}}\FM_{\Thu}.$$
Taking into account that $\FM_{\Thu}(\overline\calD\setminus\calD)=0$ we deduce that
$$\lim_{L\to\infty}\frac 1{L^{6g-6+2r}}\sum_{
\sigma\in(\Map(\Sigma)\cdot\sigma^o)\cap\calD
}\delta_{\frac 1L\sigma}=\frac{\FC(\sigma^o)}{\FB_{g,r}}\cdot(\FM_{\Thu}\vert_{\calD})$$
and that
$$\lim_{L\to\infty}\frac 1{L^{6g-6+2r}}\sum_{
\sigma\in(\Map(\Sigma)\cdot\sigma^o)\cap\overline\calD
}\delta_{\frac 1L\sigma}=\frac{\FC(\sigma^o)}{\FB_{g,r}}\cdot(\FM_{\Thu}\vert_{\overline\calD}).$$
Using once again that $\overline{\calD}\setminus\calD$ has vanishing Thurston measure, we get that these two limits actually agree. From the continuity of $\BI(\cdot,\vec\gamma^0)$ and from \eqref{ch10 equation have to go now1} we get thus that
\begin{equation}\label{ch10 equation listening to james obrien}
\lim_{L\to\infty}\overline m(\sigma^o,\vec\gamma^o,L)=\BI(\cdot,\vec\gamma^o)_*\left(\frac{\FC(\sigma^o)}{\FB_{g,r}}\cdot(\FM_{\Thu}\vert_{\calD})\right).
\end{equation}
The claim follows once we note that $B(\sigma^o)=\FC(\sigma^o)$ because \eqref{eq technical condition} implies that $\sigma^o$ has trivial stabilizer.
\end{proof}

\section{Subconvergence and disintegration of sublimits}\label{sec subconvergence}
As all along, let $\Sigma$ be a non-exceptional, compact, connected, orientable surface with genus $g$ and $r$ boundary components. Given a $k$-multicurve $\vec\gamma^o\in\CC(\Sigma)^k$ we consider as in the introduction the measures
$$m(\vec\gamma^o,L)=\frac 1{L^{6g-6+2r}}\sum_{\vec\gamma\in\Map(\Sigma)\cdot\vec\gamma^o}\delta_{\frac 1L\vec\gamma}$$
on $\CC(\Sigma)^k$. We are interested in their behavior when $L\to\infty$. Theorem \ref{main} asserts indeed that the measures $m(\vec\gamma^o,L)$ converge when $L\to\infty$ to a measure of some specific form. Our goal in this section is to prove that this is the case if we allow ourselves to pass to subsequences:

\begin{prop}\label{prop main}
Let $\vec\gamma^o=(\gamma_i)_i\in\CC(\Sigma)^k$ be a $k$-multicurve and consider any sequence $L_s\to\infty$. There is a subsequence $(L_{s_j})_j$ such that the limit 
$$\FN=\lim_{j\to\infty}m(\vec\gamma^o,L_{s_j})$$
exists in the weak-*-topology. Moreover, the limiting measure $\FN$ is the push-forward of $\FQ\otimes\FM_{\Thu}$ under the map
\begin{equation}\label{eq Delta}
\BD:\Delta_k\times\CM\CL(\Sigma)\to\CC(\Sigma)^k,\ \ ((a_i),\lambda)\mapsto(a_1\lambda,\dots,a_k\lambda)
\end{equation}
for some measure $\FQ$ on the standard $k$-simplex $\Delta_k\subset\BR_{\ge 0}^k$.
\end{prop}

Starting with the proof, fix a compact set $K\subset\Sigma\setminus\D\Sigma$ such that $\CC_K(\Sigma)$ contains the mapping class group orbit $\Map(\Sigma)\cdot\gamma_i$ of every component of $\vec\gamma$, and note that $m(\vec\gamma^o,L)$ is supported by the subset $\CC_K(\Sigma)^k=\CC_K(\Sigma)\times\dots\times\CC_K(\Sigma)$ of $\CC(\Sigma)^k$. With this notation in place we prove next that the measures $m(\vec\gamma^o,L)$ form a relatively compact family. 

\begin{lem}\label{lem precompact}
Every sequence $L_s\to\infty$ has a subsequence $(L_{s_j})_j$ such that the limit
$$\FN=\lim_j m(\vec\gamma^o,L_{s_j})$$
exists in the weak-*-topology on the space of Radon measures on $\CC_K(\Sigma)^k$. Moreover, the measure $\FN$ is mapping class group invariant.
\end{lem}
\begin{proof}
Note first that since all the measures $m(\gamma^o,L)$ are mapping class group invariant, the same will be true for any putative limit $\FN$. 

Now, since $\CC_K(\Sigma)$ is locally compact, so is the product $\CC_K(\Sigma)^k$. It follows that to prove that our sequence of measures $(m(\vec\gamma^o,L_s))_s$ has a convergent subsequence, it suffices to show that the sequence of real numbers $(m(\vec\gamma^o,L_s)(W))_s$ is bounded for every compact set $W$. Note now that the restriction to $\CC_K(\Sigma)$ of the length function $\ell:\CC(\Sigma)\to\BR_{\ge 0}$ associated to some hyperbolic metric on $\Sigma$ is proper, and hence so is also the function
$$\hat\ell:\CC_K(\Sigma)^k\to\BR_{\ge 0},\ \ \hat\ell(\vec\mu)=\sum_i\ell(\mu_i) = \ell(\pi(\vec\mu)),$$
and fix some $R$ with $\hat\ell(W)\subset[0,R]$. The restriction of $\pi:\CC(\Sigma)^k\to\CC(\Sigma)$ to $\Map(\Sigma)\cdot\vec\gamma^o$ is $c_0$-to-1 for some finite $c_0$ and it follows that
\begin{align*}
m(\vec\gamma^o,L)(W)
&\le m(\vec\gamma^o,L)\big(\{(\vec\mu)\in\CC_K(\Sigma)^k\text{ with }\hat\ell(\vec\mu)\le R\}\big)\\
&=\frac 1{L^{6g-6+2r}}\left\vert\left\{\vec\gamma\in\Map(\Sigma)\cdot\vec\gamma^o\text{ with }\hat\ell\left(\frac 1L\vec\gamma\right)\le R\right\}\right\vert\\
&=\frac {c_0}{L^{6g-6+2r}}\cdot \left\vert\left\{\pi(\vec\gamma)\in\Map(\Sigma)\cdot\pi(\vec\gamma^o)\text{ with }\ell\left(\frac 1L\pi(\vec\gamma)\right)\le R\right\}\right\vert\\
&=\frac {c_0}{L^{6g-6+2r}}\cdot \left\vert\left\{\pi(\vec\gamma)\in\Map(\Sigma)\cdot\pi(\vec\gamma^o)\text{ with }\ell(\pi(\vec\gamma))\le R\cdot L\right\}\right\vert.
\end{align*}
Now we get from \cite[Theorem 5.19]{ES book} that there is some $A>0$, depending on the multicurve $\pi(\vec\gamma^o)$ and on $R$, with 
\begin{equation}\label{eq uniform bound}
\vert\{\pi(\vec\gamma)\in\Map(\Sigma)\cdot\pi(\vec\gamma^o)\text{ with }\ell(\pi(\vec\gamma))\le R\cdot L\}\vert\le A\cdot (R\cdot L)^{6g-6+2r}
\end{equation}
for all large $L$. Having proved that $m(\vec\gamma^o,L)(W)\le c_0\cdot A\cdot R^{6g-6+2r}$ for all sufficiently large $L$, we are done.
\end{proof}

From now on let $\FN$ be as in Lemma \ref{lem precompact}. Our next goal is to show that $\FN$ is supported by the subset of $k$-multicurrents representing a measured lamination.

\begin{lem}\label{lem support ml}
The limiting measure $\FN$ is supported by $\pi^{-1}(\CM\CL(\Sigma))\subset\CC(\Sigma)^k$.
\end{lem}
\begin{proof}
With the same notation as in the proof of Lemma \ref{lem precompact} we let $\hat\ell(\vec\mu)=\sum\ell(\mu_i)$, let $c_0$ and $A$ satisfy \eqref{eq uniform bound}, and stress that there is some $K$ such that all the measures $m(\vec\gamma,L)$, and hence also the sublimit $\FN$, are supported by $\CC_K(\Sigma)^k$. Now, noting that for all $\vec\gamma\in\Map(\Sigma)\cdot\vec\gamma^o$ we have 
$$\iota(\pi(\vec\gamma),\pi(\vec\gamma))=\iota(\pi(\vec\gamma^o),\pi(\vec\gamma^o))$$
we get for all fixed but otherwise arbitrary $R>0$ that
\begin{align*}
\int_{\hat\ell^{-1}[0,R)}&\iota(\pi(\vec\mu),\pi(\vec\mu))\ d\FN(\mu)\le\\
&\le \limsup_{L\to\infty}\int_{\hat\ell^{-1}[0,R]}\iota(\pi(\vec\mu),\pi(\vec\mu))\ dm(\vec\gamma^o,L)(\mu)\\
&=\limsup_{L\to\infty}\frac 1{L^{6g-6+2r}}\sum_{\tiny{\begin{array}{l}\vec\gamma\in\Map(\Sigma)\cdot\vec\gamma^o\\ \hat\ell(\gamma)\le RL\end{array}}}\iota\left(\pi\left(\frac 1L\vec\gamma\right),\pi\left(\frac 1L\vec\gamma\right)\right)\\
&=\limsup_{L\to\infty}\left(\frac{\vert\{\vec\gamma\in\Map(\Sigma)\cdot\vec\gamma^o\text{ with }\hat\ell(\vec\gamma)\le RL\}\vert}{L^{6g-6+2r}}\cdot\frac{\iota(\pi(\vec\gamma^o),\pi(\vec\gamma^o))}{L^2}\right)\\
&\le \limsup_{L\to\infty} \left(c_0\cdot A\cdot R^{6g-6+2r}\cdot\frac{\iota(\pi(\vec\gamma^o),\pi(\vec\gamma^o))}{L^2}\right)=0.
\end{align*}
Having proved that
$$\int_{\hat\ell^{-1}[0,R)}\iota(\pi(\vec\mu),\pi(\vec\mu))\ d\FN(\mu)=0$$
for all $R$, we deduce that the sublimit $\FN$ is supported by the set of those $\vec\mu\in\CC^k(\Sigma)$ with $\iota(\pi(\vec\mu),\pi(\vec\mu))=0$. Our claim follows from Lemma \ref{lem1}.
\end{proof}

Our next goal is to refine Lemma \ref{lem support ml}, showing that $\FN$ is indeed supported by the image of the map \eqref{eq Delta}.

\begin{lem}\label{lem support Delta}
The limiting measure $\FN$ is supported by $\BD(\Delta_k\times\CM\CL(\Sigma))$ where $\BD$ is the map \eqref{eq Delta}.
\end{lem}
\begin{proof}
As in the proof of Lemma \ref{lem precompact} let $c_0$ be such that the restriction of $\pi:\CC(\Sigma)^k\to\CC(\Sigma)$ to the orbit $\Map(\Sigma)\cdot\vec\gamma^o$ is $c_0$-to-1, set $\eta^o=\pi(\vec\gamma^o)$ and note that 
$$\pi_*(m(\vec\gamma^o,L))=\frac {c_0}{L^{6g-6+2r}}\sum_{\eta\in\Map(\Sigma)\cdot\eta^o}\delta_{\frac 1L\eta}.$$
With the same notation as in \eqref{eq counting measure} this reads as 
$$\pi_*(m(\vec\gamma^o,L))=c_0\cdot m^L_{\eta^o}$$
Now, we get from the continuity of $\pi$ and \eqref{eq measure convergence} that
\begin{equation}\label{eq this cat is a pain}
\pi_*(\FN)=c_0\cdot\frac{c(\eta^o)}{\FB_{g,r}}\cdot\FM_{\Thu}.
\end{equation}
As already mentioned, it is a result of Masur \cite{Masur almost all uq} that the set $\CE$ of uniquely ergodic measured laminations has full Thurston measure. We get thus from \eqref{eq this cat is a pain} that $\pi^{-1}(\CE)$ has full $\FN$-measure. Since Lemma \ref{lem2} implies that $\pi^{-1}(\CE)$ is contained in the image of \eqref{eq Delta}, the claim follows.
\end{proof} 

We are now ready to finish the proof of Proposition \ref{prop main}:

\begin{proof}[Proof of Proposition \ref{prop main}]
Let us recap what we know so far. We get from Lemma \ref{lem precompact} that the sequence $(m(\vec\gamma^o,L_s))_s$ is precompact and that the accumulation points are mapping class group invariant. Denoting by
$$\FN=\lim_im(\vec\gamma^o,L_{s_i})$$
the limit of some convergent subsequence $(m(\vec\gamma^o,L_{s_i}))_i$ we get from Lemma \ref{lem support Delta} that $\FN$ is supported by the image of the map
$$\BD:\Delta_k\times\CM\CL(\Sigma)\to\CC(\Sigma)^k,\ \ ((a_i),\lambda)\mapsto(a_1\lambda,\dots,a_k\lambda).$$
Since this map is a homeomorphism onto its image, we can think of the limiting measure $\FN$ as being the push-forward under $\BD$ of a measure $\hat\FN$ on $\Delta_k\times\CM\CL(\Sigma)$. What is left to do is to prove that $\hat\FN$ is the product of a measure on $\Delta_k$ and the Thurston measure. 

Before doing that note however that the map $\BD$ is mapping class group equivariant, where the action of $\Map(\Sigma)$ on $\Delta_k\times\CM\CL(\Sigma)$ is trivial on the first factor and the standard one on the second. Invariance of the measure $\FN$ and equivariance of the embedding $\BD$ imply that also $\hat\FN$ is mapping class group invariant. Note also that the composition of $\BD$ with the projection map $\pi$ has a very simple form:
$$(\pi\circ\BD)((a_i),\lambda)=\lambda.$$

Coming now to the meat of the proof, suppose that $U\subset\Delta_k$ is a Borel set and consider the restriction $\hat\FN_U$ of $\hat\FN$ to $U\times\CM\CL$. The measure $\hat\FN_U$ is mapping class group invariant because both $U\times\CM\CL$ and $\hat\FN$ are. This implies that also $(\pi\circ\BD)_*(\hat\FN_U)$ is mapping class group invariant. The push-forward $(\pi\circ\BD)_*(\hat\FN_U)$ is evidently absolutely continuous with respect to $(\pi\circ \BD)_*(\hat\FN)=\pi_*(\FN)$ and hence with respect to $\FM_{\Thu}$ because of \eqref{eq this cat is a pain}. It follows thus from the ergodicity of the Thurston measure \cite{Masur ergodic} that $(\pi\circ\BD)_*(\hat\FN_U)$ is a constant multiple of $\FM_{\Thu}$. Let $\FQ(U)\in\BR_{\ge 0}$ be the unique number with
\begin{equation}\label{eq define measure}
(\pi\circ\BD)_*(\FN_U)=\FQ(U)\cdot\FM_{\Thu}
\end{equation}
It is evident that $\FQ(\emptyset)=0$ and that the map $\FQ$ from the Borel $\sigma$-algebra of $\Delta_k$ to $\BR_{\ge 0}$ is $\sigma$-additive. In other words, $\FQ$ is a measure. It remains to check that $\hat\FN$ actually agrees with $\FQ\otimes\FM_{\Thu}$. To see that this is the case note that for any $V\subset\CM\CL$ we have $(U\times\CM\CL)\cap(\pi\circ\BD)^{-1}(V)=U\times V$. This means that we have
\begin{align*}
\hat\FN(U\times V)&\stackrel{\text{definition}}=\hat\FN_U(U\times V)=(\pi\circ\BD)_*(\FN_U)(V)\\
&\stackrel{\eqref{eq define measure}}=\FQ(U)\cdot\FM_{\Thu}(V)=(\FQ\otimes\FM_{\Thu})(U\times V)
\end{align*}
as we needed to prove.
\end{proof}

\section{Proof of Theorem \ref{main}}\label{sec main}
Oh surprise, surprise, in this section we prove Theorem \ref{main}. In a nutshell, what we need to prove is that the measure $\FQ$ in Proposition \ref{prop main} does not depend on the individual sequence. First we recall some notation from the introduction: $\FQ_{\vec\gamma^o}$ is the measure on $\Delta_k$ given by
$$\FQ_{\vec\gamma^o}(U)=\frac{1}{\FB_{g,r}}\cdot\BI(\cdot,\vec\gamma^o)_*(\FM_{\Thu}\vert_{\calD_{\vec\gamma^o}})(\cone(U))$$
where
$$\BI:\CC(\Sigma)\times\CC(\Sigma)^k\to\BR_{\ge 0}^k,\ \ \ \BI:(\sigma,(\mu_i)_i)\mapsto(\iota(\sigma,\mu_i)_i),$$
where $\calD_{\vec\gamma}=\calD_{\Stab(\vec\gamma)}$ is as in Proposition \ref{prop fundamental domain}, and where $\cone(U) = \{tu \text{ with }t\in[0,1], u\in U\}$ is the cone with basis $U$. 

\begin{named}{Theorem \ref{main}}
For any $k$-multicurve $\vec\gamma^o=(\gamma_i)_i$ in $\Sigma$, we have 
\begin{equation}\label{eq main limit}
\BD_*(\FQ_{\vec\gamma^o}\otimes\FM_{\Thu})=\lim_{L\to\infty}m(\vec\gamma^o,L)
\end{equation}
in the weak-*-topology on Radon measures on $\CC(\Sigma)^k$.
\end{named}

\begin{proof}
Let us explain the basic strategy of the proof. First, we get from Proposition \ref{prop main} that there are subsequences $(L_n)$ such that the limit
\begin{equation}\label{eq main sub limit}
\FN=\lim_{n\to\infty}\frac 1{L_n^{6g-6+2r}}\sum_{\vec\gamma\in\Map(\Sigma)\cdot\vec\gamma^o}\delta_{\frac 1{L_n}\vec\gamma}
\end{equation}
exists and is of the form
\begin{equation}\label{eq getting lost}
\FN=\BD_*(\FQ\otimes\FM_{\Thu})
\end{equation}
for some measure $\FQ$ on $\Delta$. Since our family of measures is pre-compact by Lemma \ref{lem precompact}, all we need to do to prove the existence of the limit \eqref{eq main limit} is to show that the measure $\FN$ from \eqref{eq main sub limit} does not depend on the specific sequence $(L_n)$. Because of the specific form of $\FN$ it suffices to prove that the measure $\FQ$ does not depend on the sequence. To do that we will compute $\FQ$.

Once the strategy of the proof is clear, let us get into the matter. We start by fixing a filling multicurve $\sigma^o$ in $\Sigma$ satisfying \eqref{eq technical condition} and note that
$$\BI(\sigma^o,\cdot)_*\left(\frac 1{L_n^{6g-6+2r}}\sum_{\vec\gamma\in\Map(\Sigma)\cdot\vec\gamma^o}\delta_{\frac 1{L_n}\vec\gamma}\right)=\frac 1{L_n^{6g-6+2r}}\sum_{\vec\gamma\in\Map(\Sigma)\cdot\vec\gamma^o}\delta_{\frac 1{L_n} \BI(\sigma^o,\vec\gamma)}.$$
From Proposition \ref{prop special  case} we get that 
$$\BI(\cdot,\vec\gamma^o)_*\left(\frac{B(\sigma^o)}{\FB_{g,r}}\cdot(\FM_{\Thu}\vert_{\calD_{\vec\gamma^o}})\right)=\lim_{L\to\infty}\frac 1{L^{6g-6+2r}}\sum_{\vec\gamma\in\Map(\Sigma)\cdot\vec\gamma^o}\delta_{\frac 1L \BI(\sigma^o,\vec\gamma)}.$$
Now, from \eqref{eq main sub limit} and the continuity of $\BI(\sigma^o,\cdot)$ we get that
\begin{equation}\label{eq cafe1}
\BI(\sigma^o,\cdot)_*\FN=\BI(\cdot,\vec\gamma^o)_*\left(\frac{B(\sigma^o)}{\FB_{g,r}}\cdot(\FM_{\Thu}\vert_{\calD_{\vec\gamma^o}})\right).
\end{equation}
In light of \eqref{eq getting lost} we can write this as 
\begin{equation}\label{eq almost done}
\big(\BI(\sigma^o,\cdot)\circ\BD\big)_*(\FQ\otimes\FM_{\Thu})=\BI(\cdot,\vec\gamma^o)_*\left(\frac{B(\sigma^o)}{\FB_{g,r}}\cdot\FM_{\Thu}\right).
\end{equation}
To write this in a nicer way we will work in polar coordinates:
$$\polar:\BR^k_{\ge 0}\to\Delta_k\times\BR^0_{\ge 0},\ \ \polar(\vec x)=\left(\frac 1{\Vert\vec x\Vert}\vec x,\Vert\vec x\Vert\right).$$
The reason to do so is that we have
\begin{align*}
\polar\circ\BI(\sigma^o,\cdot)\circ\BD&:\Delta_k\times\CM\CL(\Sigma)\to\Delta_k\times\BR_{\ge 0}\\
\polar\circ\BI(\sigma^o,\cdot)\circ\BD&:((a_i),\lambda)\mapsto((a_i),\iota(\sigma^o,\lambda))
\end{align*}
and we then get from \eqref{eq almost done} that
$$\FQ\otimes\big(\iota(\sigma^o,\cdot)_*\FM_{\Thu}\big)=\frac{B(\sigma^o)}{\FB_{g,r}}\cdot\big(\polar\circ\BI(\cdot,\vec\gamma^o)\big)_*\left(\FM_{\Thu}\vert_{\calD_{\vec\gamma^o}}\right)$$
Evaluating both sides of this equality at the set $U\times[0,1]$ we have
$$\FQ(U)\cdot\big(\iota(\sigma^o,\cdot)_*\FM_{\Thu}\big)[0,1]=\frac{B(\sigma^o)}{\FB_{g,r}}\cdot\big(\polar\circ\BI(\cdot,\vec\gamma^o)\big)_*\left(\FM_{\Thu}\vert_{\calD_{\vec\gamma^o}}\right)(U\times[0,1])$$
All that is left is cleaning up a bit. Indeed, note that
$$\big(\iota(\sigma^o,\cdot)_*\FM_{\Thu}\big)[0,1]=\FM_{\Thu}(\{\lambda\in\CM\CL(\Sigma)\vert\iota(\sigma^o,\lambda)\le 1\})=B(\sigma^o)$$
and that
$$\big(\polar\circ\BI(\cdot,\vec\gamma^o)\big)_*\left(\FM_{\Thu}\vert_{\calD_{\vec\gamma^o}}\right)(U\times[0,1])=\BI(\cdot,\vec\gamma^o)_*(\FM_{\Thu}\vert_{\calD_{\vec\gamma^o}})(\cone(U))$$
where $\cone(U)=\polar^{-1}(U\times[0,1])=\{tu\text{ with }t\in[0,1]\text{ and }u\in U\}$ is the cone with base $U$. Combining the last three equalities we get
$$\FQ(U)\cdot B(\sigma^o)=\frac{B(\sigma^o)}{\FB_{g,r}}\cdot\BI(\cdot,\vec\gamma^o)_*(\FM_{\Thu}\vert_{\calD_{\vec\gamma^o}})(\cone(U)),$$
from where we get the identity
\begin{equation}\label{eq nail down measure}
\FQ(U)=\frac{1}{\FB_{g,r}}\cdot\BI(\cdot,\vec\gamma^o)_*(\FM_{\Thu}\vert_{\calD_{\vec\gamma^o}})(\cone(U)).
\end{equation}
We have proved that $\FQ$ is the measure given in \eqref{eq q measure}, that is $\FQ=\FQ_{\vec\gamma^o}$. In particular $\FQ$, and hence the sublimit $\FN$, does not depend on the specific sequence $(L_n)$. Summing up, we have proved that
$$\BD_*(\FQ_{\vec\gamma^o}\otimes\FM_{\Thu})=\lim_{n\to\infty}\frac 1{L_n^{6g-6+2r}}\sum_{\vec\gamma\in\Map(\Sigma)\cdot\vec\gamma^o}\delta_{\frac 1{L_n}\vec\gamma}$$
for any sequence $L_n\to\infty$ such that the limit on the right exists. Since our family of measures is pre-compact by Lemma \ref{lem precompact}, it follows that the actual limit exists. This concludes the proof of Theorem \ref{main}.
\end{proof}

We note that the measure $\FQ_{\vec\gamma^o}$ is 0 when restricted to the boundary of $\Delta_k$. Indeed, note that the preimage under $\BI(\cdot, \vec\gamma^o)$ of $\partial\Delta_k$ is the union of the sets $\{\lambda\in \CC(\Sigma) \text{ with } \iota(\lambda, \gamma_i)=0\}$ for $i=1, \ldots k$ and as we pointed out earlier, each of these sets has 0 Thurston measure. We record this for later use: 

\begin{lem}\label{lemma new}
For every $k$-multicurve $\vec\gamma^o$, we have $\FQ_{\vec\gamma^o}(\partial\Delta_k) = 0.$ \qed
\end{lem}

\section{Assymptotic distributions of length vectors}\label{sec counting}
In this section we derive Theorem \ref{main counting} from Theorem \ref{main}. First we start with a more general but weaker version of Theorem \ref{main counting}:

\begin{prop}\label{prop general counting}
With notation as in Theorem \ref{main} let $F:\CC(\Sigma)^k\to\BR_{\ge 0}^m$ be a continuous, positive and homogenous map and let
\begin{equation}\label{eq I am getting lost}
\FC(F)\stackrel{\text{def}}=\big((F\circ\BD)_*(\FQ_{\vec\gamma^o}\otimes\FM_{\Thu})\big)(\{\vec x\in\BR^k_{\ge 0}\text{ with }\Vert\vec x\Vert\le 1\}).
\end{equation}
Denoting by 
$${\bf M}_{F,\gamma^o}(L)=\{\vec\gamma\in\Map(\Sigma)\cdot\vec\gamma^o\text{ with }\Vert F(\vec\gamma)\Vert\le L\}$$
the set of all $k$-multicurrents $\vec\gamma$ in the orbit of $\vec\gamma^o$ whose image under $F$ has at most norm $1$, we have
$$\frac 1{\FC(F)}\big((F\circ\BD)_*(\FQ_{\vec\gamma}\otimes\FM_{\Thu})\big)=\lim_{L\to\infty}\frac 1{\vert{\bf M}_{F,\vec\gamma^o}(L)\vert}\sum_{\vec\gamma\in\Map(\Sigma)\cdot\vec\gamma^o}\delta_{\frac 1LF(\vec\gamma)}$$
with respect to the weak-*-topology on measures on $\BR_{\ge 0}^m$. 
\end{prop}

Here we say that the map $F$ is positive if it maps non-zero $k$-multicurrents to non-zero vectors, and we say that it is homogenous if $F(t\cdot\vec\mu)=t\cdot F(\vec\mu)$ for all $t\ge 0$ and all $\vec\mu\in\CC(\Sigma)^k$. Positivity implies that $\vert{\bf M}_{F,\vec\gamma^o}(L)\vert <\infty$ for all $L$.

\begin{proof}
From the homogeneity of $F$ we get that $F(\frac 1L\vec\gamma)=\frac 1LF(\vec\gamma)$ and hence that
\begin{equation}\label{eq getting tired}
F_*\left(\frac 1{L^{6g-6+2r}}\sum_{\vec\gamma\in\Map(\Sigma)\cdot\vec\gamma^o}\delta_{\frac 1L\vec\gamma}\right)=\frac 1{L^{6g-6+2r}}\sum_{\vec\gamma\in\Map(\Sigma)\cdot\vec\gamma^o}\delta_{\frac 1LF(\vec\gamma))}.
\end{equation}
Evaluating the right side we get that 
$$\left(F_*\left(\frac 1{L^{6g-6+2r}}\sum_{\vec\gamma\in\Map(\Sigma)\cdot\vec\gamma^o}\delta_{\frac 1L\vec\gamma}\right)\right)(\{\vec x\in\BR^k_{\ge 0}\text{ with }\Vert\vec x\Vert\le 1\})=\frac{\vert{\bf M}_{F,\vec\gamma^o}(L)\vert}{L^{6g-6+2r}}$$
Now, Theorem \ref{main} and the continuity of $F$ imply that
\begin{equation}\label{eq cute baby}
(F\circ\BD)_*(\FQ_{\vec\gamma^o}\otimes\FM_{\Thu})=\lim_{L\to\infty}\frac 1{L^{6g-6+2r}}\sum_{\vec\gamma\in\Map(\Sigma)\cdot\vec\gamma^o}\delta_{\frac 1LF(\vec\gamma))}.
\end{equation}
Moreover, homogeneity of $F$ implies that $(F\circ\BD)_*(\FQ_{\vec\gamma^o}\otimes\FM_{\Thu})$ has the same scaling behavior as the Thurston measure:
$$\big((F\circ\BD)_*(\FQ_{\vec\gamma^o}\otimes\FM_{\Thu})\big)(L\cdot U)=L^{6g-6+2r}\cdot \big((F\circ\BD)_*(\FQ_{\vec\gamma^o}\otimes\FM_{\Thu})\big)(U)$$
for all $U\subset\BR_{\ge 0}^k$. This implies that the $(F\circ\BD)_*(\FQ_{\vec\gamma^o}\otimes\FM_{\Thu})$-measure of the set $\{\Vert F(\cdot)\Vert=1\}$ vanishes and hence, by the convergence of the measures and the Portmanteau theorem, that
$$\lim_{L\to\infty}\frac{\vert{\bf M}_{F,\vec\gamma^o}(L)\vert}{L^{6g-6+2r}}=\big((F\circ\BD)_*(\FQ_{\vec\gamma^o}\otimes\FM_{\Thu})\big)(\{\vec x\in\BR^k_{\ge 0}\text{ with }\Vert\vec x\Vert\le 1\})\stackrel{\eqref{eq I am getting lost}}=\FC(F).$$
It follows thus that, up to paying the price of multiplying by the constant $\FC(F)$, we can replace in the limit in \eqref{eq cute baby} the scaling factor $L^{6g-6+2r}$ by $\vert{\bf M}_{F,\vec\gamma^o}(L)\vert$. When we do that we get
$$\frac 1{\FC(F)}\big((F\circ\BD)_*(\FQ_{\vec\gamma^o}\otimes\FM_{\Thu})\big)=\lim_{L\to\infty}\frac 1{\vert{\bf M}_{F,\vec\gamma^o}(L)\vert}\sum_{\vec\gamma\in\Map(\Sigma)\cdot\vec\gamma^o}\delta_{\frac 1L\vec\gamma}$$
This concludes the proof of Proposition \ref{prop general counting}.
\end{proof}

Shortly, when we prove Theorem \ref{main counting}, we will consider maps $\CC(\Sigma)^k\to\BR_{\ge 0}^k$ of a very specific form. Note however that Proposition \ref{prop general counting} allows for many different kinds of maps. For example, if we fix two negatively curved Riemannian metrics $X,Y$ on $\Sigma$, a conformal structure $Z$ on $\Sigma$, and a generating set $S$ for the fundamental group of $\pi_1(\Sigma)$ then we could apply the proposition to the following map:
$$F:\CC(\Sigma)^2\to\BR^4,\ \ F(\mu_1,\mu_2)=\left(\begin{array}{c}-\ell_X(\mu_1)+\Lip(Y,X)\cdot\ell_Y(\mu_1)\\ 
\ell_X(\mu_1)+\ell_X(\mu_2)-\ell_X(\mu_1+\mu_2)\\ \Ext_Z(\mu_1+\mu_2) \\ \Vert\mu_1\Vert_S+\Vert\mu_2\Vert_S-\Vert\mu_1+\mu_2\Vert_S\end{array}\right)$$
where $\ell_X(\cdot)$ and $\ell_Y(\cdot)$ are the length function associated to the Riemannian metrics $X$ and $Y$, where $\Lip(Y,X)$ is the smallest possible Lipschitz constant of maps $(\Sigma,Y)\to(\Sigma,X)$ homotopic to the identity, where $\Ext_Z(\cdot)$ is the extremal length associated to the conformal structure $Z$, and where $\Vert\cdot\Vert_S$ is the stable norm on $\pi_1(\Sigma)$ associated to the word metric given by $S$. We are not sure why anybody would want to consider such maps, but we just wanted to point out that it would be possible.
\medskip

After this comment, let us prove Theorem \ref{main counting}, but first we recall the notation used in the introduction: Given a continuous, homogenous and positive function $\phi:\CC(\Sigma)\to\BR_{\ge 0}$ and a $k$-multicurve $\vec\gamma$ we are setting
$$\BL_\phi(\vec\gamma)=\left(\frac 1{\sum_i \phi(\gamma_i)}(\phi(\gamma_1),\dots,\phi(\gamma_k)),\sum_i \phi(\gamma_i)\right)\in\Delta_k\times\BR_{\ge 0}.$$
The theorem is about the convergence of the measures 
$$\underline m(\phi,\vec\gamma^o,L)=\frac 1{\vert{\bf M}_{\phi,\vec\gamma^o}(L)\vert}\sum_{\vec\gamma\in\Map(\Sigma)\cdot\vec\gamma^o}\delta_{\BL_\phi(\frac 1L\vec\gamma)}$$
on $\Delta_k\times\BR_{\ge 0}$, where 
$${\bf M}_{\phi,\vec\gamma^o}(L)=\left\{(\gamma_i)_i\in\Map(\Sigma)\cdot\vec\gamma^o\ \middle\vert\  \sum_i\phi(\gamma_i)\le L\right\}.$$
This is what we have to prove:

\begin{named}{Theorem \ref{main counting}}
For every $k$-multicurve $\vec\gamma^o=(\gamma_1,\dots,\gamma_k)$ in $\Sigma$ there is a probability measure $\FP_{\vec\gamma^o}$ on $\Delta_k$ such that for any continuous, positive, homogenous function $\phi:\CC(\Sigma)\to\BR_{\ge 0}$ on the space of currents on $\Sigma$ we have
$$\lim_{L\to\infty}\underline m(\phi,\vec\gamma^o,L)=\FP_{\vec\gamma^o}\otimes\left((6g-6+2r)\cdot t^{6g-7+2r}\cdot d{\bf t}\right)$$
where $d{\bf t}$ stands for the standard Lebesgue measure on $\BR_{\ge 0}$. In particular, the limiting distribution is independent of $\phi$.
\end{named}
\begin{proof}
Applying Proposition \ref{prop general counting} to the map
$$F_\phi:\CC(\Sigma)^k\to\BR_{\ge 0}^k,\ \ F_\phi((\mu_i)_i)=(\phi(\mu_i)_i)$$
we get that
\begin{equation}\label{eq ciabata}
\frac 1{\FC(F_\phi)}\big((F_\phi\circ\BD)_*(\FQ_{\vec\gamma^o}\otimes\FM_{\Thu})\big)=\lim_{L\to\infty}\frac 1{\vert{\bf M}_{F_\phi,\vec\gamma^o}(L)\vert}\sum_{\vec\gamma\in\Map(\Sigma)\cdot\vec\gamma^o}\delta_{\frac 1LF_\phi(\vec\gamma)}
\end{equation}
Let us calculate $\FC(F_\phi)$:

\begin{lem}\label{lem ensaimada}
We have $\FC(F_\phi)=\Vert\FQ_{\vec\gamma^o}\Vert\cdot B(\phi)$ where, as before, $B(\phi)=\FM_{\Thu}(\{\lambda\in\CM\CL(\Sigma) \text{ with } \phi(\lambda)\le 1\})$.
\end{lem}
\begin{proof}
Noting that 
$$(F_\phi\circ\BD)((a_i)_i,\lambda)=(a_i\cdot\phi(\lambda))_i=\phi(\lambda)\cdot(a_i)_i$$
and hence that
$$\Vert(F_\phi\circ\BD)((a_i)_i,\lambda)\Vert=\phi(\lambda)\cdot\Vert(a_i)_i\Vert=\phi(\lambda)$$
we have
\begin{align*}
\FC(F_\phi)
&\stackrel{\eqref{eq I am getting lost}}=\big((F_\phi\circ\BD)_*(\FQ_{\vec\gamma^o}\otimes\FM_{\Thu})\big)(\{\vec x\in\BR^k_{\ge 0}\text{ with }\Vert\vec x\Vert\le 1\})\\
&=(\FQ_{\vec\gamma^o}\otimes\FM_{\Thu})(\{((a_i),\lambda)\in\Delta_k\times\CM\CL(\Sigma)\text{ with }\Vert(F_\phi\circ\BD)((a_i),\lambda)\Vert\le 1\})\\
&=(\FQ_{\vec\gamma^o}\otimes\FM_{\Thu})(\{((a_i),\lambda)\in\Delta_k\times\CM\CL(\Sigma)\text{ with }\phi(\lambda)\le 1\})\\
&=(\FQ_{\vec\gamma^o}\otimes\FM_{\Thu})(\Delta_k\times\{\lambda\in\CM\CL(\Sigma)\text{ with }\phi(\lambda)\le 1\})\\
&=\Vert\FQ_{\vec\gamma^o}\Vert\cdot B(\phi)
\end{align*}
as we wanted to prove.
\end{proof}

Considering again our version of polar coordinates
$$\polar:\BR_{\ge 0}^k\to\Delta_k\times\BR_{k\ge 0},\ \ \polar(\vec x)=\left(\frac 1{\Vert\vec x\Vert}\cdot\vec x,\Vert\vec x\Vert\right),$$
note that we have
$$\polar(F_\phi((\mu_i)_i))=\polar(\phi(\mu_1),\dots,\phi(\mu_k))=\BL_\phi((\mu_i)_i)$$
and hence that
$$\underline m(\phi,\vec\gamma^o,L)=\polar_*\left(\frac 1{\vert{\bf M}_{F_\phi,\vec\gamma^o}(L)\vert}\sum_{\vec\gamma\in\Map(\Sigma)\cdot\vec\gamma^o}\delta_{\frac 1LF_\phi(\vec\gamma)}\right).$$
From \eqref{eq ciabata}, Lemma \ref{lem ensaimada}, and the continuity of $\polar$ we get thus that
\begin{equation}\label{eq ciabata1}
\frac 1{\Vert\FQ_{\vec\gamma^o}\Vert\cdot B(\phi)}\big((\polar\circ F_\phi\circ\BD)_*(\FQ_{\vec\gamma^o}\otimes\FM_{\Thu})\big)=\lim_{L\to\infty}\underline m(\phi,\vec\gamma^o,L).
\end{equation}
What we have to do is to calculate the measure on the left side in \eqref{eq ciabata1} and what helps us is that the map
\begin{equation}\label{eq composition2}
\polar\circ F_\phi\circ\BD:\Delta_k\times\CM\CL(\Sigma)\to\Delta_k\times\BR_{\ge 0}
\end{equation}
is very simple:
\begin{align*}
(\polar\circ F_\phi\circ\BD)((a_i)_i,\lambda)&=(\polar\circ F_\phi)((a_i\cdot\lambda)_i)\\
&\stackrel{(*)}=\polar((a_i\cdot \phi(\lambda))_i)\stackrel{(**)}=((a_i),\phi(\lambda))
\end{align*}
where in (*) we used homogeneity of $\phi$ and in (**) the fact that $\sum_ia_i=1$ because $(a_i)_i\in\Delta_k$. In other words we have that $\polar\circ F_\phi\circ\BD=\Id\times \phi$ and hence that
$$(\polar\circ F_\phi\circ\BD)_*(\FQ_{\vec\gamma^o}\otimes\FM_{\Thu})=\Id_*(\FQ_{\vec\gamma^o})\otimes \phi_*(\FM_{\Thu})=\FQ_{\vec\gamma^o}\otimes\phi_*(\FM_{\Thu}).$$
We record what we have proved so far:
\begin{equation}\label{eq polar key}
\frac 1{\Vert\FQ_{\vec\gamma^o}\Vert\cdot B(\phi)}\cdot\left(\FQ_{\vec\gamma^o}\otimes\phi_*(\FM_{\Thu})\right)=\lim_{L\to\infty}\underline m(\phi,\vec\gamma^o,L).
\end{equation}
The next observation is that it is easy to give a formula for $\phi_*(\FM_{\Thu})$. 

\begin{lem}\label{lem calculation}
We have $\phi_*(\FM_{\Thu})=B(\phi)\cdot(6g-6+2r)\cdot t^{6g-7+2r}\cdot d{\bf t}$ where $d{\bf t}$ stands for the standard Lebesgue measure on $\BR_{\ge 0}$.
\end{lem}
\begin{proof}
Note that it suffices to show that the measures on the left and the right agree when evaluated on all intervals $[0,T]$. On the right side we get the value
$$B(\phi)\cdot(6g-6+2r)\int_0^T t^{6g-7+2r}d{\bf t}=B(\phi)\cdot T^{6g-6+2r}$$
and on the left we get
\begin{align*}
\phi_*\FM_{\Thu}[0,T]
&=\FM_{\Thu}(\{\lambda\in\CM\CL(\Sigma)\text{ with }\phi(\lambda)\le T\})\\
&\stackrel{(*)}=\FM_{\Thu}(\{\lambda\in\CM\CL(\Sigma)\text{ with }\phi(\lambda)\le 1\})\cdot T^{6g-6+2r}\\
&=B(\phi)\cdot T^{6g-6+2r}
\end{align*}
where in (*) we used once again that $\phi$ is homogenous as well as the scaling behavior of $\FM_{\Thu}$. Having found that both measures assign the same value to the interval $[0,T]$ for all $T$, we get that they both agree.
\end{proof}

Continuing with the proof of Theorem \ref{main counting} note that combining \eqref{eq polar key} and Lemma \ref{lem calculation} we get that
$$\frac 1{\Vert\FQ_{\vec\gamma^o}\Vert}\cdot\FQ_{\vec\gamma^o}\otimes\big((6g-6+2r)\cdot t^{6g-7+2r}d{\bf t}\big) = \lim_{L\to\infty}\underline m(\phi,\vec\gamma^o,L).$$
Theorem \ref{main counting} follows when we set $\FP_{\vec\gamma^o}=\frac 1{\Vert\FQ_{\vec\gamma^o}\Vert}\cdot\FQ_{\vec\gamma^o}$.
\end{proof}

\section{Asymptotic distribution of vectors of ratios}\label{sec ratios}
We now prove Theorem \ref{main ratios}. Once again we recall the notation from the introduction. Given two continuous, homogenous, positive functions $\phi,\psi:\CC(\Sigma)\to\BR_{\ge 0}$ and a $k$-multicurve $\vec\gamma$ we denote by
$$\ratio_{\psi/\phi}((\gamma_i)_i)=\left(\frac{\psi(\gamma_i)}{\phi(\gamma_i)}\right)_i\in\BR_{\ge 0}^k$$
the vector whose entries are the ratios of the $\psi$ and $\phi$ values of each component of $\vec\gamma$. What we care about is the limit of the measures 
$$r(\phi/\psi,\vec\gamma^o, L)=\frac 1{\vert{\bf M}_{\phi,\vec\gamma^o}(L)\vert}\sum_{\vec\gamma\in{\bf M}_{\phi,\vec\gamma^o}(L)}\delta_{\ratio_{\psi/\phi}(\vec\gamma)}$$
as $L\to\infty$.

\begin{named}{Theorem \ref{main ratios}}
For any two continuous, homogenous, positive functions $\phi,\psi:\CC(\Sigma)\to\BR_{\ge 0}$ there is a probability measure $\FR_{\psi/\phi}$ on $\BR_{\ge 0}$ with
$$\diag_*(\FR_{\psi/\phi})=\lim_{L\to\infty}\frac 1{\vert{\bf M}_{\phi,\vec\gamma^o}(L)\vert}\underline r(\phi/\psi,\vec\gamma^o, L)$$
for any $k$-multicurve $\vec\gamma^o$. Here $\diag:\BR_{\ge 0}\to\BR_{\ge 0}^k$ is the diagonal map $\diag(t)=(t,\dots,t)$.
\end{named}
\begin{proof}
First note that, as in the proof of Theorem \ref{main counting} we get, for every $\phi$, that
$$\vert{\bf M}_{\phi,\vec\gamma^o}(L)\vert\sim\FC(\phi)\cdot L^{6g-6+2r}=\Vert\FQ_{\vec\gamma^o}\Vert\cdot B(\phi)\cdot L^{6g-6+2r},$$
meaning that for every $\psi$ we have
$$ r(\phi/\psi,\vec\gamma^o, L)\sim\frac 1{\Vert\FQ_{\vec\gamma^o}\Vert\cdot B(\phi)}\frac 1{L^{6g-6+2r}}\sum_{\vec\gamma\in{\bf M}_{\phi,\vec\gamma^o}(L)}\delta_{\ratio_{\psi/\phi}(\vec\gamma)},$$
where we have used $\sim$ to indicate that the expressions are asymptotic as $L\to\infty$. It follows that it suffices to consider the measures on the right side.

Well, note that the map which sends the $k$-multicurve $\vec\gamma$ to $\ratio_{\psi/\phi}(\vec\gamma)$ extends continuously to the space of $k$-multicurrents with non-zero entries:
$$\ratio_{\psi/\phi}:(\CC(\Sigma)\setminus\{0\})^k\to \BR_{\ge 0}^k,\ \ 
\ratio_{\psi/\phi}((\mu_i)_i)=\left(\frac{\psi(\mu_i)}{\phi(\mu_i)}\right)_i$$
and that this map has the following properties:
\begin{align}
\ratio_{\phi/\psi}(L\cdot\vec\mu)&=\ratio_{\phi/\psi}(\vec\mu),\text{ and}   \label{equ1}\\  
(\ratio_{\psi/\phi}\circ\BD)((a_i),\lambda)&=\diag\left(\frac{\psi(\lambda)}{\phi(\lambda)}\right).  \label{equ2}
\end{align}
From \eqref{equ1} we get that
\begin{align*}
\frac 1{L^{6g-6+2r}}\sum_{\vec\gamma\in{\bf M}_{\phi,\vec\gamma^o}(L)}&\delta_{\ratio_{\psi/\phi}(\vec\gamma)}=\\
&=\frac 1{L^{6g-6+2r}}\sum_{\vec\gamma\in{\bf M}_{\phi,\vec\gamma^o}(L)}\delta_{\ratio_{\psi/\phi}(\frac 1L\vec\gamma)}\\
&=(\ratio_{\psi/\phi})_*\left(\frac 1{L^{6g-6+2r}}\sum_{\vec\gamma\in{\bf M}_{\phi,\vec\gamma^o}(L)}\delta_{\frac1L\vec\gamma}\right)\\
&=(\ratio_{\psi/\phi})_*\left(m(\vec\gamma,L)\vert_{\{\vec\mu\in\CC(\Sigma)^k\text{ with }\Vert\phi(\mu)\Vert\le 1\}}\right)
\end{align*}
where $m(\vec\gamma,L)$ is the measure considered in Theorem \ref{main}. The convergence of the measures $m(\vec\gamma,L)$, the continuity of $\ratio_{\psi/\phi}$ on its domain $(\CC(\Sigma)\setminus\{0\})^k$, and the fact that $\FQ_{\vec\gamma^o}(\partial\Delta_k)=0$ (Lemma \ref{lemma new}) imply thus that
\begin{align*}
\lim_{L\to\infty}\frac 1{L^{6g-6+2r}}&\sum_{\vec\gamma\in{\bf M}_{\phi,\vec\gamma^o}(L)}\delta_{\ratio_{\psi/\phi}(\vec\gamma)}=\\
&=(\ratio_{\psi/\phi})_*\left((\BD_*(\FQ_{\vec\gamma^o}\otimes\FM_{\Thu}))\vert_{\{\mu\in\CC(\Sigma)^k\vert\phi(\mu)\le 1\}}\right)\\
&=(\ratio_{\psi/\phi}\circ\BD)_*\left((\FQ_{\vec\gamma^o}\otimes\FM_{\Thu})\vert_{\Delta_k\times\{\lambda\in\CM\CL(\Sigma)\vert \phi(\lambda)\le 1\}}\right).
\end{align*}
Using \eqref{equ2}, we can rewrite this as 
\begin{multline*}
\lim_{L\to\infty}\frac 1{L^{6g-6+2r}}\sum_{\vec\gamma\in{\bf M}_{\phi,\vec\gamma^o}(L)}\delta_{\ratio_{\psi/\phi}(\vec\gamma)}=\\
=\Vert\FQ_{\vec\gamma^o}\Vert\cdot\left(\diag\circ\left(\frac{\psi(\cdot)}{\phi(\cdot)}\right)\right)_*\left(\FM_{\Thu}\vert_{\{\lambda\in\CM\CL(\Sigma)\vert\phi(\lambda)\le 1\}}\right).
\end{multline*}
Scaling by the missing factors we get thus that
\begin{align*}
\lim_{L\to\infty} r(\phi/\psi,\vec\gamma^o, L)
&=\lim_{L\to\infty}\frac 1{\Vert\FQ_{\vec\gamma^o}\Vert\cdot B(\phi)}\frac 1{L^{6g-6+2r}}\sum_{\vec\gamma\in{\bf M}_{\phi,\vec\gamma^o}(L)}\delta_{\ratio_{\psi/\phi}(\vec\gamma)}\\
&=\diag_*\left(\left(\frac{\psi(\cdot)}{\phi(\cdot)}\right)_*\left(\frac 1{B(\phi)}\FM_{\Thu}\vert_{\{\lambda\in\CM\CL(\Sigma)\vert\phi(\lambda)\le 1\}}\right)\right)
\end{align*}
and the claim follows when we set
\begin{equation}\label{eq concrete r measure}
\FR_{\psi/\phi}=\left(\frac{\psi(\cdot)}{\phi(\cdot)}\right)_*\left(\frac 1{B(\phi)}\FM_{\Thu}\vert_{\{\lambda\in\CM\CL(\Sigma)\vert\phi(\lambda)\le 1\}}\right).
\end{equation}
We have proved Theorem \ref{main ratios}.
\end{proof}

\section{Examples}\label{sec examples}
In this final section we discuss the measures $\FP_{\vec\gamma^o}$ and $\FR_{\psi/\phi}$ in some concrete examples. In a nutshell we get that the measure $\FP_{\vec\gamma^o}$ is much better behaved than the measure $\FR_{\psi/\phi}$. In our examples this is made apparent by the fact that while we calculate $\FP_{\vec\gamma^o}$ in some concrete cases, all results about $\FR_{\psi/\phi}$ describe different pathologies.

\subsection*{Some pathological examples of the measure $\FR_{\psi/\phi}$}
From \eqref{eq concrete r measure} we get an explicit description of the measure $\FR_{\psi/\phi}$ for two given continuous, homogenous and positive functions $\phi,\psi:\CC(\Sigma)\to\BR_{\ge 0}$. We want however a different formula. First note that $\phi$ determines a probability measure on the space $\CP\CM\CL(\Sigma)$ of projective measured laminations via the formula
$$\FN_\phi(U)=\frac 1{B(\phi)}\FM_{\Thu}(\{\lambda\in\CM\CL(\Sigma)\text{ with }\phi(\lambda)\le 1\text{ and }[\lambda]\in U\})$$
where $[\lambda]$ stands for the projective class of $\lambda$. Noting that the function $\frac{\psi(\cdot)}{\phi(\cdot)}$ descends to a well-defined map
$$\frac{\psi(\cdot)}{\phi(\cdot)}:\CP\CM\CL(\Sigma)\to\BR_{\ge 0}$$
it is easy to see that we can rewrite $\FR_{\psi/\phi}$ as
\begin{equation}\label{eq push r}
\FR_{\psi/\phi}=\left(\frac{\psi(\cdot)}{\phi(\cdot)}\right)_*\FN_{\phi}.
\end{equation}
Since $\CP\CM\CL(\Sigma)$ is compact and connected, we get that the support of $\FR_{\psi/\phi}$ is a possibly degenerate interval $[a,b]\subset\BR_{>0}$. Our first observation, a not very surprising one, is that $\FR_{\psi/\phi}$ can well be absolutely continuous with respect to Lebesgue measure:

\begin{bei}\label{ex lebesgue class}
{\em There are two filling weighted multicurves $\alpha,\beta$ such that $\FR_{\iota(\beta,\cdot)/\iota(\alpha,\cdot)}$ is absolutely continuous with respect to the Lebesgue measure.} To construct our two multicurves we start by choosing a filling multicurve $\sigma$ and a collection $\gamma_1,\dots,\gamma_n$ of curves with the property that whenever we have $\lambda,\lambda'\in\CM\CL(\Sigma)$ with
$$\iota(\gamma_i,\lambda)=\iota(\gamma_i,\lambda')\text{ for all }i=1,\dots,n$$
then $\lambda=\lambda'$. This means in particular that the map
$$\BI = \BI(\cdot, (\gamma_i)):\{\lambda\in\CM\CL(\Sigma),\ \iota(\sigma,\lambda)=1\}\to\BR^n,\ \ \BI(\lambda)=(\iota(\gamma_i,\lambda))_i$$
is an embedding. In fact, if we endow $\CM\CL(\Sigma)$ with its standard PL-structure, then $\BS$ is a PL-homeomorphism onto a PL-sphere in Euclidean space. Moreover, if we identify $\CP\CM\CL(\Sigma)\simeq\{\lambda\in\CM\CL(\Sigma),\ \iota(\sigma,\lambda)=1\}$ by sending the class $[\lambda]$ to its unique representative $\lambda$ with $\iota(\sigma,\lambda)=1$, then $\BI_*(\FN_{\iota(\sigma,\cdot)})$ is a measure in the Lebesgue class associated to the PL-structure. It follows that for generic choices of $(a_1,\dots,a_n)\in\BR^n_{\geq0}$ the push forward of $\BI_*(\FN_{\iota(\sigma,\cdot)})$ under the linear map $\rho_{(a_1,\dots,a_n)}:\BR^n\to\BR$ given by $\rho_{(a_1,\dots,a_n)}(x_1,\dots,x_n)\mapsto\sum a_ix_i$ is a measure which is absolutely continuous with respect to the Lebesgue measure. Now, the point of all this is that 
$$\frac{\iota(\sigma+\sum a_i\gamma_i,\cdot)}{\iota(\sigma,\cdot)}:\CP\CM\CL(\Sigma)\simeq\{\lambda\in\CM\CL(\Sigma),\ \iota(\sigma,\lambda)=1\}\to \BR$$
satisfies 
$$\frac{\iota(\sigma+\sum a_i\gamma_i,\cdot)}{\iota(\sigma,\cdot)}=1 + (\rho_{(a_1,\dots,a_n)}\circ\BI).$$
We can thus take $\alpha=\sigma$ and $\beta=\sigma+\sum a_i\gamma_i$ for generic choices of $a_i>0$.
\end{bei}

In some sense Example \ref{ex lebesgue class} is just an example of what one would expect generically to happen. It is however clear that $\FR_{\psi/\phi}$ is not always absolutely continuous with respect to the Lebesgue measure. A silly example of this would be to take $\phi=\psi$ because in this case our measure is a Dirac measure $\delta_1$ centered at $1\in\BR_{\ge 0}$. There are however much more interesting examples:

\begin{bei}
{\em There are two distinct filling multicurves $\alpha,\beta$ such that $\FR_{\iota(\alpha,\cdot)/\iota(\beta,\cdot)}=\delta_1$ is a Dirac measure.} Indeed, Horowitz \cite{Horowitz} proved that there are plenty of pairs of curves $\alpha,\beta$ which have the same length in any hyperbolic metric and it is easy to get them filling: just get two in a pair of pants and then take an immersion of the pair of pants in the surface $\Sigma$ in such a way that each curve in the pair of pants is sent to a filling curve. Anyways, in \cite{Leininger03} Leininger proved that any one of Horowitz's pairs of curves $\alpha,\beta$ also satisfies that $\iota(\alpha,\lambda)=\iota(\beta,\lambda)$ for all $\lambda\in\CM\CL(\Sigma)$. It follows that the map $\ratio_{\iota(\alpha,\cdot)/\iota(\beta,\cdot)}$ is constant $1$. The claim follows thus from \eqref{eq push r}.
\end{bei}

To construct more pathological examples recall that, when we choose a hyperbolic metric on $\Sigma$, the space of currents can be identified with the space $\CM_{flip-flow}(\Sigma)$ of measures on the unit tangent bundle $T^1\Sigma$ which are invariant under the geodesic flip and the geodesic flow (see for example \cite[Theorem 3.4]{ES book}). It follows that if $f:T^1\Sigma\to\BR_+$ is any continuous function then 
$$\phi_f:\CC(\Sigma)\simeq\CM_{flip-flow}(\Sigma)\to\BR_+,\ \phi_f(\mu)=\int_{T^1\Sigma}f\ d\mu$$
is a continuous, homogenous and positive function on the space of currents. Now, the key observation is that, because of the Birman-Series Theorem \cite{Birman-Series}, the set $\CX$ of those vectors in $T^1\Sigma$ which are tangent to some simple geodesic is a closed set of Hausdorff dimension $1$. It follows that there are plenty of continuous functions $f:T^1\Sigma\to\BR_+$ with support disjoint of $\CX$ and for any such we have $\phi_f(\lambda)+\phi(\lambda)=\phi(\lambda)$ for all $\lambda\in\CM\CL(\Sigma)$ and any $\phi:\CC(\Sigma)\to\BR_{\ge 0}$. We thus have the following:

\begin{bei}
{\em There are plenty of pairs of distinct continuous, homogenous, positive functions $\phi,\psi:\CC(\Sigma)\to\BR_+$ with $\psi(\lambda)=\phi(\lambda)$ for all $\lambda\in\CM\CL(\Sigma)$. For any such pair we have $\FR_{\phi,\psi}=\delta_1$.}
\end{bei}

Now one can use the same idea to construct all sorts of examples. For instance, let $\lambda_1,\lambda_2\in\CM\CL(\Sigma)$ be two maximal measured laminations whose supports are mutually transversal geodesic laminations and let $U_1,U_2\subset T^1\Sigma$ be open neighborhoods of the supports of $\lambda_1$ and $\lambda_2$ respectively. Then let $K_i\subset U_i$ be a compact neighborhood of the support of $\lambda_i$ and let $f:T^1\Sigma\to\BR_+$ be a continuous function with $f\vert_{K_1}$ constant $2$ and $f\vert_{K_2}$ constant $3$. The map $\ratio_{\phi_f/\phi_{\bf 1}}$ takes then the value $2$ in a neighborhood of $\lambda_1$ and $3$ in a neighborhood of $\lambda_2$. It follows that $\FR_{\phi_f/\phi_{\bf 1}}$ has at least two atoms, one at $2$ and one at $3$.

Evidently one can use this construction to get any number of atoms. In fact, even to get infinitely many atoms. And clearly one has an enormous flexibility when constructing $f$.

\begin{bei}\label{example r atoms}
{\em For any $k=1,2,\dots,\infty$ there are plenty of pairs of distinct continuous, homogenous and positive functions $\phi,\psi:\CC(\Sigma)\to\BR_{\ge 0}$ such that $\FR_{\psi/\phi}$ has $k$ atoms.}
\end{bei}

It is evident that in general the measure $\FR_{\psi/\phi}$ can be pretty badly behaved and that it seems that one can get any pathology one can think of. It might be however interesting to see how $\FR_{\psi/\phi}$ behaves if we restrict $\phi,\psi$ to belonging to some natural class, such as length functions of hyperbolic metrics.

\subsection*{Some calculations of $\FP_{\vec\gamma^o}$}
Let us recall now that the measure $\FP_{\vec\gamma^o}$ is the probability measure on $\Delta_k$ proportional to the measure \eqref{eq q measure} given, up to a constant, by
\begin{equation}\label{eq loud kids}
U\mapsto \BI(\cdot,\vec\gamma^o)_*(\FM_{\Thu}\vert_{\calD_{\vec\gamma^o}})(\cone(U))
\end{equation}
where $\calD_{\vec\gamma^o} = \calD_{\Stab(\vec\gamma^o)}$ is a fundamental domain for the action of $\Stab(\vec\gamma^o)$ on $\CM\CL(\Sigma)$, and where as always $\cone(U)$ is the cone with basis $U$. It is thus clear that to calculate $\FP_{\vec\gamma^o}$ one needs to understand 
\begin{enumerate}
\item how to build a fundamental domain $\calD_{\vec\gamma^o}$ for the action $\Stab(\vec\gamma^o)\actson\CM\CL(\Sigma)$, and
\item the restriction of the map $\BI(\cdot,\vec\gamma^o)$ to $\calD_{\vec\gamma^o}$.
\end{enumerate}
Well, there are good news and bad news. The good news is that this is in principle feasible, and the bad news is that as soon as $\vec\gamma$ gets complicated then anybody agreeing to do it must be either slightly unaware of the amount of work involved, or under the influence. Let us still consider some specific cases:

\begin{bei}
Let $\Sigma$ be a one-holed torus, let $\alpha^o,\beta^o$ be simple curves which intersect once, and set $\vec\gamma^o=(\alpha^o,\beta^o)$. The map
$$\BI(\cdot,\vec\gamma^o):\CM\CL(\Sigma)\to\BR_{\ge 0}^2$$
is 2-to-1 on $\BR_{>0}^2$ and 1-to-1 elsewhere. From here, and the definition of the Thurston measure, we get that
$$\BI(\cdot,\vec\gamma^o)_*\FM_{\Thu}=2\cdot\Leb_{\BR^2}$$
where $\Leb_{\BR^2}$ is the standard Lebesgue measure. We thus get from \eqref{eq loud kids} that
$$\FP_{\vec\gamma^o}(U)=\frac 1{\sqrt 2}\Leb_{\Delta_2}(U)$$
for every $U\subset\Delta_2$ measurable. 
\end{bei}

In the previous example and in the sequel, we denote by $\Leb_{\Delta_k}$ the standard Lebesgue measure on $\Delta_k$, that is the one obtained by integrating the Riemannian volume form. 

\begin{bei}\label{example pants}
Suppose that $\vec P^o$ is a labeled pants decomposition. The stabilizer of $\vec P^o$ is nothing other than the group $\BT\subset\Map(\Sigma)$ consisting of multi-twists along the components of $\vec P^o$ and it is easy to find a fundamental domain $\calD_{\vec P^o}$ for the action of $\BT$ on $\CM\CL(\Sigma)$: it is intuitively the set of measured laminations whose leaves twist at most once around the components of $\vec P^o$, and this can be made formal using Dehn-Thurston coordinates. In those coordinates, the map $\BI(\cdot,\vec P^o)$ is just the projection onto half of the coordinates. We get then from there that
$$\BI(\cdot,\vec P^o)_*(\FM_{\Thu}\vert_{\calD_{\vec P^o}})(V)=\frac 1{2^{2g-3+r}}\int_V\prod x_i\ dx_1\ldots dx_{3g-3+r}$$
for any $V\subset\BR_{\ge 0}^{3g-3+r}$ (see for example \cite[Exercise 12.1]{ES book}, although this computation is in one way or another already in Mirzakhani's thesis \cite{Maryam thesis}, and we would not be surprised if there were earlier sources) meaning that 
\begin{align*}
\FQ_{\vec P^o}(U) & = \frac{1}{\FB_{g,r}}\cdot \frac 1{2^{2g-3+r}}\int_{\cone(U)}\prod x_i\ dx_1\ldots dx_{3g-3+r}\\
&= \frac{1}{\FB_{g,r}}\cdot \frac 1{2^{2g-3+r}}\cdot \frac{1}{2(3g-3+r)^{3/2}}\int_{U}\prod x_i\ d\Leb_{\Delta_{3g-3+r}}
\end{align*}
for any $U\subset \Delta_{3g-3+r}$. 

Now, taking into account that $\FP_{\vec P^o} = \frac {\FQ_{\vec P^o}}{\Vert\FQ_{\vec P^o}\Vert}$ and that 

\begin{equation*}
\int_{\Delta_{3g-3+r}} \prod x_i \,d\Leb_{\Delta_{3g-3+r}} = %\\
%= \frac{\vol(\Delta_{3g-3+r})}{\vol(\cone(\Delta_{3g-4+r}))}  \int_{\cone(\Delta_{3g-4+r})}x_1\cdots x_{3g-4+r}\left(1-\sum x_i\right)\,dx_1\cdots dx_{3g-4+r}\\
 \frac{\sqrt{3g-3+r}}{(6g-7+2r)!}
\end{equation*}
we get that 
$$\FP_{\vec P^o}(U) = \frac{(6g-7+2r)!}{\sqrt{3g-3+r}}\cdot\int_U \prod x_i\ d\Leb_{\Delta_{3g-3+r}}$$
for all $U\subset \Delta_{3g-3+r}$.

This concludes the discussion for the pair of pants.
\end{bei}  

The real reason why one can calculate $\FQ_{\vec P^o}$ for a labeled pants decomposition is that it is not hard to give a specific fundamental domain for the action of $\Stab(\vec P^o)$ on $\CM\CL(\Sigma)$. Explicit fundamental domains can also be given for general simple multicurves $\vec\gamma^o$, and we could thus use them to calculate $\FQ_{\vec\gamma^o}$ in such cases. But one can also do something else. Indeed, relying on the earlier mentioned results of Arana-Herrera \cite{Arana} and Liu \cite{Liu} we get from Corollary \ref{kor counting} the following:

\begin{bei}\label{example simple multicurves}
If $\vec\gamma^o$ is a simple $k$-multicurve then we have for every $U\subset\Delta_k$
$$\FP_{\vec\gamma^o}(U)=\int_U P_{\vec\gamma^o}d\Leb_{\Delta_k}$$
for some specific homogenous polynomial $P_{\vec\gamma^o}$ of degree $6g-6+2r-k$.
\end{bei}

\begin{bem}
From Arana-Herrera \cite{Arana} and Liu \cite{Liu} one gets an expression for the coefficients of the polynomial $P_{\vec\gamma^o}$ in terms of intersection numbers of Chern classes of line bundles on the Deligne-Mumford compactification of moduli space. If on the other hand one uses the fundamental domain implicit in the proof of \cite[Theorem 11.2]{ES book} then the coefficients of $P_{\vec\gamma^o}$ are expressed in terms of Thurston volumes of the set of measured laminations carried by some specific train tracks. Using work of Norbury \cite{Norbury} and Kontsevich \cite{Kontsevich}, one can reconcile the two points of view along the lines of what is done in \cite{LRST}.
\end{bem}

In all these examples the measure is in the Lebesgue class and has full support. This is evidently not true in general, basically because there is no reason for the map $\BI(\cdot,\vec\gamma^o):\CM\CL(\Sigma)\to\BR_{\ge 0}^k$ to be surjective. Indeed, this map is PL with respect to the standard PL-structure of $\CM\CL(\Sigma)$ and this implies that it cannot possibly be surjective if $k>6g-6+2r=\dim(\CM\CL(\Sigma))$. Anyways, since the action $\Stab(\vec\gamma^o)\actson\CM\CL(\Sigma)$ also admits a PL-fundamental domain, we get the following:

\begin{prop}\label{prop describe p measure}
For every $k$-multicurve $\vec\gamma^o=(\gamma_1,\dots,\gamma_k)$ there are finitely many linear maps $L_1,\dots,L_n:\BR^{6g-6+2r}\to\BR^k$ mapping $\Delta_{6g-6+2r}$ into $\Delta_k$ such that $\FP_{\vec\gamma}$ is a linear combination with positive coefficients of the measures $(L_1)_*\Leb_{\Delta_{6g-6+2r}},\dots,(L_n)_*\Leb_{\Delta_{6g-6+2r}}$.
\end{prop}

\begin{proof}[Sketch of proof]
We will only prove Proposition \ref{prop describe p measure} in the case that each component $\gamma_1,\dots,\gamma_k$ is a non-simple curve and that $\sum\gamma_i$ has trivial stabilizer. This fully suffices to get the gist of the proof and the reader used to working with train-tracks will have no difficulty filling in the details of the remaining cases. Well, fixing a hyperbolic metric on $\Sigma$ we get from our additional assumptions on $\vec\gamma^o$ that there is some $\epsilon$ so that all the intersections between a leaf of a measured lamination and one of the components $\gamma_i$ happen with at angle greater than $4\epsilon$. Take now finitely many maximal train-tracks which carry all of $\CM\CL(\Sigma)$ and such that the set of measured laminations carried by two of them has negligible measure (take for example the standard train tracks from \cite{Penner-Harer}). Now, splitting these train tracks finitely many times we can replace this initial collection by a new collection $\tau_1,\dots,\tau_n$ of train tracks which besides having these same properties, also have geodesic curvature $<\epsilon$ and meet the components of $\vec\gamma^o$ with angle greater than $4\epsilon$. This implies that if we identify the set $\CM\CL(\tau_i)$ of measured laminations carried by $\tau_i$ with the set $W(\tau_i)$ of positive solutions of the switch equations, then the restriction of $\BI(\cdot,\vec\gamma^o)$ to $\CM\CL(\tau_i)$ is linear. The claim follows because $\CM\CL(\tau_i)=W(\tau_i)$ is simply the intersection of a linear space with the positive quadrant and because the Thurston measure on $\CM\CL(\tau_i)$ is proportional to the standard Lebesgue measure on the polyhedral cone $W(\tau_i)$.
\end{proof}

Now, the reader probably suspects that if we consider sequences $(\vec\gamma_n)_n$ of $k$-multicurves which become more and more complicated, then the measures $\FP_{\vec\gamma_n}$ can become much wilder. This is true:

\begin{bei}\label{example p complicated}
Suppose that $\Sigma$ is a one-holed torus, let $\alpha^o,\beta^o$ be two simple curves which intersect once, fix a complete hyperbolic metric $X$ on the interior of $\Sigma$, let $(\eta_n^o)_n$ be a sequence of weighted geodesics whose associated geodesic flow invariant measures converge to the Liouville measure of $X$, and set
$$\vec\gamma_n^o=(\alpha^o,\beta^o,\eta_n^o).$$ 
The key observation is that the maps $\BI(\cdot,\vec\gamma_n^o)$ converge uniformly on compacta to the map
$$\BI_\infty:\CM\CL(\Sigma)\to\BR^3,\ \ \lambda\mapsto(\iota(\alpha^o,\lambda),\iota(\beta^o,\lambda),\ell_X(\lambda))$$
and hence that measures $\BI(\cdot,\vec\gamma_n^o)_*\FM_{\Thu}$ converge to $(\BI_\infty)_*\FM_{\Thu}$. From the work of McShane--Rivin \cite{McShane-Rivin} we get that the support of $(\BI_\infty)_*\FM_{\Thu}$ is the cone over the union of two continuous maps $[0,1]\to\Delta_3$, each one of which fails to be smooth in a dense set of points. We deduce that there is no $N$ such that all measures $\FP_{\vec\gamma_n^o}$ are supported by $N$ linear images of the simplex $\Delta_2$. 
\end{bei}

\end{document}